\documentclass[a4paper,reqno]{amsart}

\usepackage{graphicx}
\usepackage{mathrsfs}
\usepackage{amsfonts}
\usepackage{amssymb}
\usepackage{amsmath}
\usepackage{amsthm}
\usepackage{amscd}
\usepackage[all,2cell]{xy}
\usepackage{color}
\usepackage[pagebackref,colorlinks]{hyperref}
\usepackage{enumerate}
\usepackage{bm}
\usepackage{bbm}
\usepackage[normalem]{ulem}
\usepackage{mathrsfs}

\usepackage{comment}
\usepackage{hyperref}
\UseAllTwocells \SilentMatrices

\numberwithin{equation}{section}

\usepackage{comment}

\theoremstyle{plain}
\newtheorem{thm}{Theorem}[section]
\newtheorem{prop}[thm]{Proposition}
\newtheorem{cor}[thm]{Corollary}
\newtheorem{lemma}[thm]{Lemma}
\newtheorem{conj}[thm]{Conjecture}

\theoremstyle{definition}
\newtheorem{deff}[thm]{Definition}
\newtheorem{example}[thm]{Example}

\theoremstyle{remark}
\newtheorem{rmk}[thm]{\bf Remark}

\def\g{\gamma}
\def\d{\delta}
\def\G{\Gamma}
\def\mG{\mathcal{G}}
\def\xra{\xrightarrow[]{}}
\def\a{\alpha}
\def\b{\beta}
\def\D{\mathcal{D}}

\def\N{\mathbb N}

\def\HH{\mathcal{H}}

\newcommand{\Ima}{\operatorname{Im}}

\def\sub{\subseteq}

\def\VV{\mathcal{V}}
\def \Z{\mathbb Z}
\def\-{\text{-}}

\newcommand{\supp}{\operatorname{supp}}
\newcommand{\Iso}{\operatorname{Iso}}

\newcommand{\Ker}{\operatorname{Ker}}

\newcommand{\gr}{\operatorname{gr}}

\newcommand{\M}{\mathbb M}

\newcommand\Gr[1][]{{\operatorname{{\bf Gr}^{#1}-}}}

\newcommand\Modd[1][]{{\operatorname{{\bf Mod}^{#1}-}}}

\newcommand{\id}{\operatorname{id}}

\newcommand{\note}[1]{\textcolor{blue}{$\Longrightarrow$ #1}}

\begin{document}

\title[Homology of \'etale groupoids]{Homology of \'etale groupoids\\a graded approach}

\author{Roozbeh Hazrat}
\address{
Centre for Research in Mathematics\\
Western Sydney University\\
Australia} \email{r.hazrat@westernsydney.edu.au, h.li@westernsydney.edu.au}

\author{Huanhuan Li}



\date{\today}

\begin{abstract} We introduce a graded homology theory for graded \'etale groupoids. 
For $\Z$-graded groupoids, we establish an exact sequence relating the graded zeroth-homology to non-graded one. Specialising to the arbitrary graph groupoids, we prove that the graded zeroth homology group with constant coefficients $\Z$  is isomorphic to the graded Grothendieck group of the associated Leavitt path algebra. To do this, we consider the diagonal algebra of the Leavitt path algebra of the covering graph of the original graph and construct the group isomorphism directly. Considering the trivial grading, our result extends Matui's on zeroth homology of finite  graphs with no sinks (shifts of finite type) to all arbitrary graphs. We use our results to show that graded zeroth-homology group is a complete invariant for eventual conjugacy of shift of finite types and could be the unifying invariant for the analytic and the algebraic graph algebras. 
\end{abstract}

\maketitle

\section{Introduction}

The homology theory for \'etale groupoids was introduced by Crainic and Moerdijk~\cite{crainicmoerdijk} 
who showed that the homology groups are invariant under Morita equivalences of \'etale groupoids and established  spectral sequences which used for the computation of these groups. Matui \cite{matui,matui15,matui16} considered this homology theory in relation with the dynamical properties of groupoids and their full groups.   In~\cite{matui} Matui proved, using Lindon-Hochschild-Serre spectral sequence established by Crainic and Moerdijk, that for an \'etale groupoid $\mG$ arising from shifts of finite type, the homology groups $H_0(\mG)$ and $H_1(\mG)$ coincide with $K$-groups $K_0(C^*(\mG))$ and $K_1(C^*(\mG))$, respectively.  Here $C^*(\mG)$ is the groupoid $C^*$-algebra associated to $\mG$ which was first systematically studied by Renault in his seminal work~\cite{renault}. The current emerging picture is that the homology of groupoids should govern the behaviour of their associated algebras and dynamics. 

A groupoid $\mG$ is called \emph{graded} if one can naturally partition it by an index group $\Gamma$, namely, $\mG$ equipped with cocycle $c:\mG\rightarrow \Gamma$. The majority of the important \'etale groupoids arising from combinatorial or dynamical data are naturally graded. Taking into account the partitions and their rearrangements should give salient information about the groupoid and the structures associated to them.  

The theory of graded groupoids relates to the theory of graded rings via the vehicle of Steinberg algebras. These algebras are the algebraic version of groupoid $C^*$-algebras which have been recently introduced in~\cite{cfst,st}. For an ample groupoid $\mG$, the Steinberg algebra $A_R(\mG)$, with arbitrary coefficient ring $R$ with unit, is an algebra which  encompasses very interesting algebras such as inverse semigroup rings, Leavitt path algebras and their higher rank versions namely Kumjian-Pask algebras~\cite{pchr}. If the groupoid $\mG$ is $\Gamma$-graded then its associated Steinberg algebra $A_R(\mG)$ is naturally $\Gamma$-graded ring.  The graded structure of these algebras, in return, allows us to consider the graded invariants associated to these algebras such as graded $K$-theory, $K_n^{\gr}$, $n\geq 0$ \cite{grbook}. 

In this note, parallel to the theory of graded $K$-theory for graded rings,  we initiate and study graded homology theory
 for a $\G$-graded \'etale groupoid $\mG$, introducing the groups 
$H_n^{\gr}(\mG)$, $n\geq 0$,  which are naturally equipped with the structure of $\G$-modules. We prove that the graded homology of a strongly graded ample groupoid is isomorphic to the homology of its $\varepsilon$-th component with $\varepsilon$ the identity of the grade group. For $\mathbb Z$-graded ample groupoid $\mG$ we establish an exact sequence
\begin{align*}
    H_0^{\gr}(\mG) \longrightarrow  H_0^{\gr}(\mG) \longrightarrow  H_0(\mG) \longrightarrow 0,
\end{align*} 
which is similar to the van den Bergh exact sequence for graded $K$-theory of $\Z$-graded Noetherian regular rings (\cite[\S6.4]{grbook}).

Specialising to the graph groupoids which have a natural $\mathbb Z$-grading, we prove that for the graph groupoid $\mG_E$, arising from an arbitrary graph $E$, we have an order preserving $\Z[x,x^{-1}]$-module isomorphism $$H_0^{\gr}(\mG_E) \cong K_0^{\gr} (L_R(E)),$$ where $L_R(E)$ is the Leavitt path algebra associated to $E$ with the coefficient field $R$.  Considering the grade group to be trivial, since the $K_0$-group of Leavitt and $C^*$-algebras associated to a graph coincide we obtain that 
\begin{equation*}
H_0(\mG_E) \cong K_0 (A(\mG_E)) \cong K_0 (L(E)) \cong K_0 (C^*(E))\cong K_0 (C^*(\mG_E)).
\end{equation*}
Since shifts of finite type are characterised by finite graphs with no sinks~\cite{lindmarcus} (see also \cite[\S5]{matui16}), our result extends Matui's result to all arbitrary graphs. This is done, not by using Lindon-Hochschild-Serre spectral sequence, but by directly using the description of the (graded) monoid of graph algebras~\cite{lpabook,ahls}.  

This approach also allows us to relate the graded homology groups to invariants of symbolic dynamics.  Namely, 
for a two-sided shift $X$, denote by $X^{+}$ the one-sided shift associated to $X$, and by $\mG_{X^{+}}$ the groupoid associated to $X^{+}$ (see~\cite{carlseneilers} for a summary of concepts on symbolic dynamics). Then we establish that  $H_0^{\gr}(\mG_X^+)$ is a complete invariant characterising eventual conjugacy of $X$ (Theorem~\ref{colmorgen}).

Throughout the note the emphasis is on the module structure of the graded homology groups which carries substantial information. 
We believe that graded homology theory allows us to unify the invariants proposed for the classification of analytic and algebraic graph algebras, namely Leavitt path algebras and $C^*$-graph algebras. Combing previous work on this direction, it is plausible to formulate the following conjecture  (see~\cite{roozbehhazrat2013} and Theorem~\ref{sojoy}). 

\begin{conj}\label{conji1}
Let $E$ and $F$ be finite graphs and $R$ a field. Then the following are equivalent.

\begin{enumerate}[\upshape(1)]

\item There is a gauge preserving isomorphism $\phi: C^*(E) \rightarrow  C^*(F)$;

\smallskip

\item There is a graded ring isomorphism $\phi:L_R(E) \rightarrow L_R(F)$;

\smallskip

\item There is an order preserving $\Z[x,x^{-1}]$-module isomorphism 
$H_0^{\gr}(\mG_E) \rightarrow H_0^{\gr}(\mG_F)$ such that $\phi([1_{\mG_E^0}])=[1_{\mG_F^0}])$. 

\end{enumerate}

\end{conj}

This note is primarily concerned with the graded ample groupoids, parallel to graded ring theory. Three essential concepts in the study of graded ring theory are strongly graded rings, graded matrix rings and smash products~\cite{grbook,nastasesu-vanoystaeyen}. We recall them in Section~\ref{gradedring}. We develop and study these concepts in the setting of graded ample groupoids in Section~\ref{grgroupoidlabel}. In Section~\ref{subsection24}, we will then consider the Steinberg algebras of these groupoids and reconcile these constructions with the parallel concepts in graded ring theory. The comparison between these two theories lead us to the definition of graded homology for graded groupoids in Section~\ref{bellesebasatian}. In fact these correspondences allow us to compute the graded homology groups for certain groupoids, such as graph groupoids in Section~\ref{sectionsix}.

\section{Graded rings} \label{gradedring}

We briefly review three main concepts in graded ring theory, namely, grading on matrices, graded Morita theory and the smash product. We also recall the case of strongly graded rings. In Section~\ref{grgroupoidlabel} we do the parallel constructions in the setting of topological groupoids.  We will then observe that the concept of Steinberg algebras will tie together these concepts. They will be used to define graded homology theory in Section~\ref{bellesebasatian} and calculate them for graph groupoids in Section~\ref{sectionsix}. 

\subsection{Graded rings}
\label{grringslablel}

Let $\Gamma$ be a group with identity $\varepsilon$. A ring $A$ (possibly without unit)
is called a \emph{$\Gamma$-graded ring} if $ A=\bigoplus_{ \gamma \in \Gamma} A_{\gamma}$
such that each $A_{\gamma}$ is an additive subgroup of $A$ and $A_{\gamma}  A_{\delta}
\subseteq A_{\gamma\delta}$ for all $\gamma, \delta \in \Gamma$. The group $A_\gamma$ is
called the $\gamma$-\emph{homogeneous component} of $A.$ When it is clear from context
that a ring $A$ is graded by group $\Gamma,$ we simply say that $A$ is a  \emph{graded
ring}. If $A$ is an algebra over a ring $R$, then $A$ is called a \emph{graded algebra}
if $A$ is a graded ring and $A_{\gamma}$ is a $R$-submodule for any $\gamma \in \Gamma$.
A $\G$-graded ring $A=\bigoplus_{\g\in\G}A_{\g}$is called \emph{strongly graded} if
$A_{\g}A_{\delta}=A_{\g\delta}$ for all $\g,\delta$ in $\G$.

The elements of $\bigcup_{\gamma \in \Gamma} A_{\gamma}$ in a graded ring $A$ are called
\emph{homogeneous elements} of $A.$ The nonzero elements of $A_\gamma$ are called
\emph{homogeneous of degree $\gamma$} and we write $\deg(a) = \gamma$ for $a \in
A_{\gamma}\backslash \{0\}.$ The set $\Gamma_A=\{ \gamma \in \Gamma \mid A_\gamma \not =
0 \}$ is called the \emph{support}  of $A$. We say that a $\Gamma$-graded ring $A$ is
\emph{trivially graded} if the support of $A$ is the trivial group
$\{\varepsilon\}$---that is, $A_\varepsilon=A$, so $A_\gamma=0$ for $\gamma \in \Gamma
\backslash \{\varepsilon\}$. Any ring admits a trivial grading by any group. If $A$ is a
$\G$-graded ring and $s \in A$, then we write $s_\alpha, \alpha \in \G$ for the unique
elements $s_\alpha \in A_\alpha$ such that $s = \sum_{\alpha \in \G} s_\alpha$. Note that
$\{\alpha \in \G : s_\alpha \not= 0\}$ is finite for every $s \in A$.

We say a $\G$-graded ring $A$ has \emph{graded local units} if for any finite set of
homogeneous elements  $\{x_{1}, \cdots, x_{n}\}\subseteq A$, there exists a homogeneous
idempotent $e\in A$ such that $\{x_{1}, \cdots, x_{n}\}\subseteq eAe$. Equivalently, $A$
has graded local units, if $A_\varepsilon$ has local units with $\varepsilon$ the identity of $\G$ and $A_\varepsilon
A_{\g}=A_{\g}A_\varepsilon=A_{\g}$ for every $\g \in \G$.

For a $\G$-graded ring $A$ with graded local units, we denote by $\Gr A$ the category of unital graded right $A$-modules with morphisms preserving the grading. For a graded right $A$-module $M$, we define the $\a$-\emph{shifted} graded right 
$A$-module $M(\a)$ as
\begin{equation*}M(\a)=\bigoplus_{\g\in \G}M(\a)_{\g},
\end{equation*}
where $M(\a)_{\g}=M_{\a\g}$. That is, as an ungraded module, $M(\alpha)$ is a copy of
$M$, but the grading is shifted by $\alpha$. For $\a\in\G$, the \emph{shift functor}
\begin{equation*}
\mathcal{T}_{\a}: \Gr A\longrightarrow \Gr A,\quad M\mapsto M(\a)
\end{equation*}
is an isomorphism with the property $\mathcal{T}_{\a}\mathcal{T}_{\b}=\mathcal{T}_{\a\b}$
for $\a,\b\in\G$.

\subsection{Graded matrix rings}\label{frankairport}
Let $A=\bigoplus_{\sigma\in\G}A_{\sigma}$ be a $\G$-graded ring with $\G$ a group, $n$ is a positive integer, and $\M_n(A)$ the ring
of $n \times n$-matrices with entries in $A$. Fix some $\overline{\sigma} = (\sigma_1, \cdots , \sigma_n) \in \G^n$. To any
$\g \in \G$, we associate the following additive subgroup of $\M_n(A)$
\begin{equation*}
\M_n(A)_{\g}(\overline{\sigma})=
\begin{pmatrix}
A_{\sigma_1\g\sigma_1^{-1}}& A_{\sigma_1\g\sigma_2^{-1}}&\cdots &A_{\sigma_1\g\sigma_n^{-1}}\\
A_{\sigma_2\g\sigma_1^{-1}}&A_{\sigma_2\g\sigma_2^{-1}}&\cdots &A_{\sigma_2\g\sigma_n^{-1}}\\
\vdots&\vdots&&\vdots\\
A_{\sigma_n\g\sigma_1^{-1}}&A_{\sigma_n\g\sigma_2^{-1}}&\cdots &A_{\sigma_n\g\sigma_n^{-1}}\\
\end{pmatrix}
\end{equation*} The family of additive subgroups $\big \{\M_n(A)_{\g}(\overline{\sigma})\;|\;  \g\in \G \big \}$ defines a $\G$-grading
of the ring $\M_n(A)$ (see \cite[Proposition 2.10.4]{nastasesu-vanoystaeyen}). We will denote this graded ring by $\M_n(A)(\overline{\sigma})$.

We denote by $\M_{\G}(A)$ the ring of $|\G|\times|\G|$-matrices with finitely many nonzero entries. Here, $|\G|$ is the cardinality of $\G$. To any $\g\in \G$, we associate the following additive subgroup of $\M_{\G}(A)$ 
 $$\M_\G(A)_{\g}=\{(x_{\a\b})_{\a,\b\in\G}\;|\; x_{\a\b}\in A_{\a\g\b^{-1}}\}.$$ The family of additive subgroups $\{\M_\G(A)_\g\;|\; \g\in\G\}$ defines a $\G$-grading for the ring $\M_\G(A)$, denoted by $\M_\G(A)(\G)$. 

\subsection{Graded Morita theory}\label{coincigftr}

We say a functor $F: \Gr A \xra \Gr B $ is a \emph{graded} functor provided $F$ commutes with the shift functors, namely $F(M(\alpha))=F(M)(\alpha)$, for every $\alpha \in \Gamma$ and $M\in \Gr A $. A graded functor $F: \Gr A \xra \Gr B $ is said to be a \emph{graded equivalence} provided there is a graded functor $G: \Gr B \xra \Gr A$ which is the inverse to $F$. When such a graded equivalence exists, we say that $A$ and $B$ are \emph{graded Morita equivalent rings}.

We also need to recall the definition of \emph{graded Morita context} from \cite[Definition 2.5]{haefner}. Let $A$ and $B$ be $\G$-graded rings with graded local units. The six tuple $(A, B, M, N, \psi: M\otimes_B N \xra A, \varphi: N\otimes_A M\xra B)$ is a \emph{graded Morita context} provided:
\begin{itemize}
\item[(1)] $_AM_B$ and $_BN_A$ are graded bimodules in the following sense:
$$A_\alpha M_\beta B_ \g \subseteq M_{\alpha\beta\g} \text{~~~~and~~~~} B_{\alpha}N_\beta A_\g \subseteq N_{\alpha\beta\g},$$ for all $\alpha, \beta, \g\in \G$.

\smallskip 

\item[(2)]  $\psi: M\otimes_B N \xra A$ and $\varphi: N\otimes_A M\xra B$ are graded homomorphisms in the sense that $$\psi(M_{\alpha}\otimes N_{\beta})\sub A_{\alpha\beta} \text{~~~~and ~~~~} \varphi(N_{\alpha}\otimes M_{\beta})\sub B_{\alpha\beta}$$ for all $\alpha, \beta\in\G$
 and that satisfy the following two conditions: $$\psi(m\otimes n)m'=m\varphi(n\otimes n') \text{~~~~and~~~~} \varphi(n\otimes m)n'=n\psi(m\otimes n'),$$
for all $m,m'\in M$ and $n,n'\in N$.
 \end{itemize}
 
 By \cite[Theorem 2.6]{haefner} the graded rings $A$ and $B$ are graded Morita equivalent if and only if there is a Morita context with $\psi$ and $\phi$ surjective. This will be used in the setting of Steinberg algebras (Theorem~\ref{hfhgwolfsheim}).

\subsection{The smash product of a graded ring}\label{onboardtodenver}

For a $\G$-graded ring $A$, the smash product ring $A \#\G$ was first defined by Cohen and Montgomery for a unital ring $A$ and a finite group $\Gamma$ in their seminal paper~\cite{cm1984} (see also~\cite[Chapter 7]{nastasesu-vanoystaeyen}). They established that the category of graded $A$-modules are equivalent to the category of $A\#\G$-modules. In turn,
this allowed them to relate the graded structure of $A$ to non-graded properties of $A$.

Recall that for a $\G$-graded ring $A$ (possibly without unit), the \emph{smash product} ring $A\#\G$
is defined as the set of all formal sums $\sum_{\gamma \in \G} r^{(\gamma)} p_\gamma $,
where $r^{(\gamma)}\in A$ and $p_\gamma$ are symbols. Addition is defined component-wise
and multiplication is defined by linear extension of the rule
\begin{equation}\label{morningfaero}
(rp_{\a})(sp_{\b})=rs_{\a\b^{-1}}p_{\b},
\end{equation}  where $r,s\in A$ and $\a,\b\in\G$.

In \cite[Proposition 2.5]{ahls}, the Cohen-Montgomery result on the equivalences of categories of graded modules of unital rings graded by finite groups extended as follows:  Let $A$ be a $\G$-graded ring with graded local units. Then there is an isomorphism of categories 
\begin{equation}\label{cohenmont}
 \Gr A\longrightarrow \Modd A\#\G.
\end{equation}

The Equation~\eqref{cohenmont} allows us to describe the graded functors for the graded ring $A$ as a corresponding non-graded version on the category of $A\#\G$. As an example we can describe the graded Grothendieck group of $A$ as a usual $K_0$-group of $A\#\G$, namely, $K^{\gr}_0(A)\cong K_0(A\#\G)$.

\section{Graded Groupoids} \label{grgroupoidlabel}

In this section we develop the concepts of graded matrix groupoids and strongly graded groupoids. We then observe that semi-direct product of groupoids will replace the concept of smash product in the classical ring theory.

\subsection{Graded groupoids}\label{subsection21}

A groupoid is a small category in which every morphism is invertible. It can also be viewed as a generalisation of a group which has partial binary operation.  Let $\mG$ be a
groupoid. If $g\in\mG$, $s(g)=g^{-1}g$ is the \emph{source} of $g$ and $r(g)=gg^{-1}$ is
its \emph{range}. The pair $(g_1,g_2)$ is composable if and only if $r(g_2)=s(g_1)$. The set
$\mG^{(0)}:=s(\mG)=r(\mG)$ is called the \emph{unit space} of $\mG$. Elements of
$\mG^{(0)}$ are units in the sense that $gs(g)=g$ and $r(g)g=g$ for all $g \in \mG$. The \emph{isotropy group} at a unit $u$ of $\mG$ is the group ${\rm Iso}(u) =\{g\in\mG \mid s(g)=r(g)=u\}$. Let ${\rm Iso}(\mG)=\bigsqcup_{u\in\mG^{(0)}}{\rm Iso}(u)$. For
$U,V\in\mG$, we define
\[
    UV=\big \{g_1 g_2 \mid g_1\in U, g_2\in V \text{ and } r(g_2)=s(g_1)\big\}.
\]

A topological groupoid is a groupoid endowed with a topology under which the inverse map
is continuous, and such that composition is continuous with respect to the relative
topology on $\mG^{(2)} := \{(g_1,g_2) \in \mG \times \mG : s(g_1) = r(g_2)\}$ inherited from
$\mG\times \mG$. An \emph{\'etale} groupoid is a topological groupoid $\mG$ such that the
domain map $s$ is a local homeomorphism. In this case, the range map $r$ is also a local
homeomorphism. In this article, by an \'etale groupoid we mean a locally compact Haudsorff groupoid such that the source map $s$ is a local homeomorphism.

An \emph{open bisection} of $\mG$ is an open subset $U\subseteq \mG$ such
that $s|_{U}$ and $r|_{U}$ are homeomorphisms onto an open subset of $\mG^{(0)}$. If $\mG$ is a \'etale groupoid, then there is a
base for the topology on $\mG$ consisting of open bisections with compact closure. As demonstrated in \cite{cfst,st}, if $\mG^{(0)}$ is totally disconnected and $\mG$ is \'etale, then there is a basis for the topology on $\mG$ consisting of compact open bisections. We say
that an \'etale groupoid $\mG$ is \emph{ample} if there is a basis consisting of compact open bisections for its topology.

Let $\G$ be a discrete group and $\mG$ a topological groupoid. A $\G$-grading of $\mG$ is
a continuous function $c : \mG \to \G$ such that $c(g_1)c(g_2) = c(g_1g_2)$ for all $(g_1,g_2)
\in \mG^{(2)}$; such a function $c$ is called a \emph{cocycle} on $\mG$. For $\g \in \G$, setting $\mG_\g=c^{-1}(\g)$ which are clopen subsets of $\mG$, we have 
$\mG=\bigsqcup_{\g \in \G}  \mG_\g$ such that $\mG_\b \mG_\g \subseteq \mG_{\b\g}$.
 In this case, we call $\mG$ a $\G$-graded groupoid.  In this paper,
we shall also refer to $c$ as the \emph{degree map} on $\mG$. 
For $\g\in \G$, we say that $X\subseteq \mG$ is $\g$-graded if $X\subseteq \mG_\g$. We
always have $\mG^{(0)} \subseteq \mG_\varepsilon$, so $\mG^{(0)}$ is
$\varepsilon$-graded where $\varepsilon$ is the identity of $\G$. We write $B^{\rm co}_{\g}(\mG)$ for the collection of all
$\g$-graded compact open bisections of $\mG$ and
\begin{equation}\label{hdgftdggd}
B_{*}^{\rm co}(\mG)=\bigcup_{\g\in\G} B^{\rm co}_{\g}(\mG).
\end{equation}

A groupoid is trivially graded by considering the map $c:\mG \rightarrow \{\varepsilon\}$.  Throughout this paper, we consider the groupoids to be $\G$-graded. As soon as our grade group $\G$ is trivial, then we obtain the results in the non-graded setting. Many important examples of \'etale groupoids have a natural graded structure.

\begin{example}[{\it Transformation groupoids}]\label{transgo}
Let $\phi:\Gamma\curvearrowright X$ be 
an action of a countable discrete group $\Gamma$ 
on a Cantor set $X$ by homeomorphisms, equivalently there is a group homomorphism $\G\xra {\rm Home}(X)$, where ${\rm Home}(X)$ consists of homeomorphisms from $X$ to $X$ with the multiplication given by the composition of maps. Let $\mG_\phi=\Gamma\times X$ and define the groupoid structure: 
$(\gamma, \gamma' x')\cdot(\gamma',x')=(\gamma\gamma',x')$, 
and $(\gamma,x)^{-1}=(\gamma^{-1}, \gamma x)$. 
Then $\mG_\phi$ is an \'etale groupoid,  
called the transformation groupoid 
arising from $\phi:\Gamma\curvearrowright X$. 
The unit space $\mG_\phi^{(0)}$ is canonically identified with $X$ 
via the map $(\varepsilon,x)\mapsto x$. The natural cocyle $\mG_\phi\rightarrow \G, (\gamma,x)\mapsto \gamma$ makes $\mG_\phi$ a $\Gamma$-graded groupoid. 
\end{example}

\begin{example}\label{exmsemidirect}
Let $\phi:\Gamma\curvearrowright \mG$ be an action of a discrete group $\G$ on an \'etale groupoid $\mG$, i.e., there is a group homomorphism $\alpha: \G\xra {\rm Aut}(\mG)$ with ${\rm Aut}(\mG)$ the group of homeomorphisms between 
$\mG$ which respect the composition of $\mG$. The \emph{semi-direct product} $\mG\rtimes_{\phi}\G$ is $\mG\times\G$ with the following groupoid structure: $(g, \g)$ and $(g',\g')$ are composable if and only if $g$ and $\phi_{\g}(g')$ are composable, $$(g, \g)(g',\g')=(g\phi_{\g}(g'), \g\g'),$$ and $$(g,\g)^{-1}=(\phi_{\g^{-1}}(g^{-1}), \g^{-1}).$$ The unit space of $\mG\rtimes_{\phi}\G$ is $\mG^{(0)}\times \{\varepsilon\}$. The source map $s:\mG\rtimes_{\phi}\G\xra \mG^{(0)}\times \{\varepsilon\}$ is given by $s(g,\g)=(s(\phi_{\g^{-1}}(g)), \varepsilon)$ and the range map $r:\mG\rtimes_{\phi}\G\xra \mG^{(0)}\times \{\varepsilon\}$ is given by $r(g,\g)=(r(g), \varepsilon)$ for $g\in\mG$ and $\g\in\G$. The semi-direct product $\mG\rtimes_{\phi}\G$ has a natural $\G$-grading given by the cocycle $c: \mG\rtimes_{\phi}\G\xra\G$,  $c(g,\g)=\g$. 

If $\mG$ is an ample groupoid and $\phi:\Gamma\curvearrowright \mG$ is an action of a discrete group $\G$ on $\mG$, then 
$\mG\rtimes_{\phi}\G$  is again ample under the product topology on $\mG\times \G$. We will determine the Steinberg algebras of these construction in~\S\ref{intheairfaroe}. 

\end{example}

\begin{example}[{\it Groupoid of a dynamical system}]\label{dynsysgr}
Let $X$ be a locally compact Hausdorff space and $\sigma:X\rightarrow X$ a local homeomorphism. Further suppose that $d:X\rightarrow \Gamma$ is a continuous map. We then have a $\G$-graded groupoid defined as follows:
\[
\mG(X,\sigma) := \Big \{ (x,s,y) \mid \sigma^n(x)=\sigma^m(y) \text{ and } s=\prod_{i=0}^{n-1}d(\sigma^i(x)) \prod_{j=0}^{m-1}d(\sigma^j(y))^{-1}, n,m \in \N \Big \}.
\]
Here the cocyle is defined by $\mG(X,\sigma) \rightarrow \G, (x,s,y)\mapsto s$. 

Let $E$ be a finite graph with no sinks and sources, and $E^\infty$ the set of all infinite paths in $E$.  Let $\sigma: E^\infty \rightarrow E^\infty; e_1e_2e_3\dots \mapsto e_2e_3\dots$ and the constant map $d:E^\infty \rightarrow \Z, x\mapsto 1$. Then we have $\mG(E^\infty,\sigma)=\mG_E$, where $\mG_E$ is the graph groupoid associated to the graph $E$ (see Section \ref{sectionsix}).

\end{example}

\subsection{Graded matrix groupoids}\label{subsection22} 
Recall from \S\ref{frankairport} that if a ring $A$ is $\G$-graded, then for a collection $\{\g_1,\dots,\g_n\}$ of elements of $\G$, we have a $\G$-graded matrix ring  
$\M_n(A)(\g_1,\dots,\g_n)$.

In this section we develop the parallel concept in the setting of groupoids, namely graded matrix groupoids with a given shift. We will see that the Steinberg algebra of a graded matrix groupoid coincides with the graded matrix ring of the Steinberg algebra with the same shift (see Proposition ~\ref{stegrrecon}).


Let $\mG$ be an \'etale groupoid.
For $f:\mG^{(0)}\xra \Z$ a map with $f\geq 0$, where $\Z$ is the discrete abelian group of integers,
we let $$\mG_f=\{x_{ij}\;|\; x\in \mG, \, 0\leq i\leq f(r(x)), 0\leq j\leq f(s(x))\},$$ and equip $\mG_f$ with the induced topology from the product topology on $\mG\times \Z\times \Z$. We endow $\mG_f$ with the groupoid structure as follows:
$$\mG_f^{(0)}=\{x_{ii}\;|\; x\in \mG^{(0)}, 0\leq i\leq f(x)\},$$ $(x_{ij})^{-1}=(x^{-1})_{ji}$, two elements are $x_{ij}$ and $y_{kl}$ are composable if and only if $s(x)=r(y)$, $j=k$ and the product is given by $x_{ij}y_{jl}=(xy)_{il}$. We call $\mG_f$ the \emph{matrix groupoid} of $\mG$ with respect to $f$. 

Suppose that $\mG$ is a $\G$-graded groupoid with the grading map $c:\mG\xra \G$ and $f:\mG^{(0)}\xra\Z$ satisfies $f\geq 0$. Fix a continuous map $\psi: \mG_f^{(0)}\xra \G$. Define  a map 
\begin{align}
\label{gradingmap}
c_{\psi}: \mG_f & \longrightarrow \G,\\
 x_{ij}& \longmapsto \psi(r(x)_{ii})c(x)\psi(s(x)_{jj})^{-1}. \notag
\end{align} 
We claim that $c_{\psi}$ is a continuous cocycle for $\mG_f$ and thus $\mG_f$ is $\G$-graded. Indeed, we observe that $c_{\psi}$ is the composition map 
$$\mG_f \stackrel{\mu}\longrightarrow\mG_f^{(0)}\times \G\times \mG_f^{(0)} \stackrel{\nu}\longrightarrow\G,$$ where $\mu(x_{ij})=(r(x)_{ii}, c(x), s(x)_{jj})$ for $x_{ij}\in\mG_f$ and $\nu(y_{kk}, \g, z_{ll})=\psi(y_{kk})\g\psi(z_{ll})^{-1}$ for $y_{kk}, z_{ll}\in\mG_f^{(0)}$ and $\g\in\G$. Since $\psi$, the source and range maps for the groupoid $\mG_f$ are continuous, $\mu$ and $\nu$ are continuous. Thus $c_{\psi}$ is continuous. Since the grading map $c_{\psi}$ of the $\G$-graded groupoid $\mG_f$ is related to the continuous map $\psi$, we denote the $\G$-graded groupoid $\mG_f$ by $\mG_f(\psi)$.

For a locally compact Hausdorff space $X$ and a toplogical group $G$, we denote by $C_c(X, G)$ the set of $G$-valued continuous functions with compact support.

For an \'etale groupoid $\mG$, we denote $s^{-1}(X)\cap r^{-1}(X)$ by $\mG|_{X}$ for a subset $X\sub \mG^{(0)}$. A subset $F\sub \mG^{(0)}$ is said to be \emph{$\mG$-full}, if $r^{-1}(x)\cap s^{-1}(F)$ is not empty for any $x\in \mG^{(0)}$.

The following lemma is the graded version of \cite[Lemma 4.3]{matui}.

\begin{lemma} \label{gradedversion}Let $\mG$ be an \'etale $\G$-graded groupoid, graded by the cocyle $c:\mG\xra \G$, whose unit space is compact and totally disconnected and let $Y\sub \mG^{(0)}$ be a $\mG$-full clopen subset. There exist $f\in C_c(Y, \Z)$ and a continuous map $\psi: (\mG|_Y)_f^{(0)}\xra \G$ such that $\pi: (\mG|_Y)_f\xra \mG$ satisfying $\pi(x_{00})=x$ for all $x\in \mG|_Y$ is a graded isomorphism. 
\end{lemma}

\begin{proof} Let $X=\mG^{(0)}$. There are compact open bisections $V_1, V_2, \cdots, V_n$ such that $s(V_j)\sub Y$ for each $j=1, \cdots, n$,  $r(V_1), r(V_2), \cdots, r(V_n)$ are mutually disjoint and their union is equal to $X\setminus Y$ (see \cite[Lemma 4.3]{matui}). For each subset $\lambda\sub \{1,2, \cdots, n\}$ we fix a bijection $\alpha_{\lambda}: \{k\;|\;1\leq k\leq |\lambda|\}\xra \lambda$. For $y\in Y$ put $$\lambda(y)=\{k\in \{1,2, \cdots, n\}\;|\; y\in s(V_k)\}.$$ We define $f\in C_c(Y, \mathbb{Z})$ by $f(y)=|\lambda(y)|$. Since each $s(V_k)$ is clopen, $f$ is continuous. A map $\theta: (\mG|_Y)_f^{(0)}\xra \mG$ is defined by 
\begin{equation*}
\theta(y_{ii})=
\begin{cases} y, & \text{if~} i=0;\\
(s|_{V_l})^{-1}(y) & \text{otherwise},
\end{cases} 
\end{equation*}
where $l=\alpha_{\lambda(y)}(i)$ (see \cite[Lemma 4.3]{matui}).

We claim that $\theta:(\mG|_Y)_f^{(0)}\xra \mG$ defined above is continuous. Recall that the collection of all the bisections of the \'etale groupoid $\mG$ forms a basis 
for its topology (see \cite[Proposition 3.5]{exel2008} and \cite[Proposition 3.4]{st}). It suffices to show that for any open bisection $T$ of $\mG$, $\theta^{-1}(T)$ is open. Observe that $$\theta^{-1}(T)=\theta^{-1}(T\cap \cup_{j=1}^nV_j) \cup (T\cap Y)_{00},$$ where $(T\cap Y)_{00}=\{x_{00}\;|\; x\in T\cap Y\}$. Indeed, clearly 
$$\theta^{-1}(T\cap \cup_{j=1}^nV_j) \cup (T\cap Y)_{00}\sub \theta^{-1}(T).$$ 
On the other hand, take any $y_{ii}\in \theta^{-1}(T)$. If $i=0$, then $\theta(y_{00})=y\in T\cap Y$, implying $y_{00}\in (T\cap Y)_{00}$. If $i\neq 0$, then $\theta(y_{ii})=(s|_{V_l})^{-1}(y)\in V_l$ with $l=\alpha_{\lambda(y)}(i)$. Since $\theta(y_{ii})$ belongs to $T$, we have $\theta(y_{ii})\in T\cap \cup_{j=1}^nV_j$, implying $y_{ii}\in \theta^{-1}(T\cap \cup_{j=1}^nV_j)$. Thus $$\theta^{-1}(T)\sub\theta^{-1}(T\cap \cup_{j=1}^nV_j) \cup (T\cap Y)_{00}.$$ Since $(T\cap Y)_{00}$ is an open subset of $(\mG|_Y)_f^{(0)}$, it suffices to show that $\theta^{-1}(T\cap \cup_{j=1}^nV_j)$ is an open subset of $(\mG|_Y)_f^{(0)}$. In fact, we have 
\begin{equation*}
\begin{split}
\theta^{-1}(T\cap \cup_{j=1}^nV_j)
=\big((T\cap \cup_{j=1}^nV_j)\cap Y\big)_{00}
\sqcup\sqcup_{j=1}^n s(T\cap V_j)_{11}
\end{split}
\end{equation*} with $s(T\cap V_j)_{11}=\{z_{11}\;|\; z\in s(T\cap V_j)\}$, since for each element $x\in T\cap \cup_{j=1}^nV_j$,  we have $f(s(x))=1$. 

 Let $\psi$ be the composition map $(\mG|_Y)_f^{(0)}\xrightarrow[]{\theta} \mG\xrightarrow[]{c}\G$. By \eqref{gradingmap}, $\mG_f$ is $\G$-graded. By the proof of  \cite[Lemma 4.3]{matui}, there is an isomorphism $\pi: (\mG|_Y)_f\xra \mG$ such that $\pi(x_{ij})=\theta(r(x)_{ii})\cdot g\cdot \theta(s(x)_{jj})^{-1}$. It is evident that the isomorphism $\pi: (\mG|_Y)_f\xra \mG$ preserves the grading. The proof is completed.
\end{proof}

Let $\G$ and $\Z$ be discrete groups, $\mG$ a $\G$-graded groupoid with $c: \mG\xra \G$ the grading map and  $f\in C_c(\mG^{(0)}, \Z)$. Suppose that $\sigma:\Z\xra \Z$ is a bijection such that for any $x\in \mG^{(0)}$, $\sigma (D_x)=D_x$ where $D_x=\{0, 1,\cdots, f(x)\}$. Suppose that $\psi: \mG_f^{(0)}\xra \G$ is a continuous map. Let $\psi_{\sigma}:\mG_f^{(0)}\xra \G$ be the continuous map given by $\psi_{\sigma}(x_{ii})=\psi(x_{\sigma(i)\sigma(i)})$ for any $x_{ii}\in \mG_f^{(0)}$.

In the following theorem, by $Z(\Gamma)$ we mean the centre of the group $\Gamma$. 

\begin{prop} Keep the above notation. We have 

\begin{enumerate}[\upshape(1)]
\item$\mG_f(\psi)\cong\mG_f(\psi_{\sigma})$ as $\G$-graded groupoids;
\smallskip

\item Suppose that $f:\mG^{(0)}\xra \Z$ is a constant function such that $f(x)=n$ with $n $ a positive integer for any $x\in \mG^{(0)}$. Let $\tau$ be a permutation of $\{0, 1, \cdots, n\}$. Let $\psi:\mG_f^{(0)}\xra \G$ be a continuous map and $\psi_{\tau}:\mG_f^{(0)}\xra \G$ be given by $\psi_{\tau}(x_{ii})=\psi(x_{\tau(i)\tau(i)})$ for any $x_{ii}\in \mG_f^{(0)}$. Then $\mG_f(\psi)\cong\mG_f(\psi_{\tau})$ as $\G$-graded groupoids;

\smallskip

\item Let $\g\in Z(\G)$ and $\psi: \mG_f^{(0)}\xra \G$ a continuous map. Let $\psi_{\g}:\mG_f^{(0)}\xra \G$ be the composition of $\psi$ and the multiplication map $\G\xra\G, \beta\mapsto \g\beta$ with $\beta\in \G$. Then $\mG_f(\psi)\cong\mG_f(\psi_{\g})$ as $\G$-graded groupoids.
\end{enumerate}
\end{prop}

\begin{proof} (1) Observe that $\mG_f(\psi)\xra \mG_f(\psi_{\sigma}), x_{ij}\mapsto x_{\sigma^{-1}(i)\sigma^{-1}(j)}$ is the desired $\G$-graded groupoid isomorphism. For (2), we observe that in this case $D_x=\{0, 1, \cdots, n\}=$ for all $x\in\mG^{(0)}$. Let $\sigma:\Z\xra \Z$ be a  bijection such that $\sigma(\{0, 1, \cdots, n\})=\{0, 1, \cdots, n\}$. Applying (1), the result follows directly. For (3), $\mG_f(\psi)\xra \mG_f(\psi_{\g}), x_{ij}\mapsto x_{ij}$ is a $\G$-graded groupoid isomorphism. 
\end{proof}

\subsection{Strongly graded groupoids}\label{stgrgroupids}

Let $\G$ be a group with identity $\varepsilon$ and let $\mG$ be a $\G$-graded groupoid, graded by the cocycle $c: \mG\xra \G$. Parellel to the classical graded ring theory, we say $\mG$ is a {\it strongly $\G$-graded groupoid} if $\mG_\alpha \mG_\beta=\mG_{\alpha \beta}$  for all $\alpha,\beta \in \G$. In \cite{clarkhazratrigby} such groupoids were systematically studied. It was  shown that  $\mG$ is strongly graded if and only if $s(\mG_\alpha)=\mG^{(0)}$ if and only if $r(\mG_\alpha)=\mG^{(0)}$, for all $\alpha \in \Gamma$ (see \cite[Lemma 2.3]{clarkhazratrigby}). When $\mG$ is strongly graded, the $\varepsilon$-component of $\mG$, namely, $\mG_\varepsilon$, contains substantial information about the groupoid $\mG$. In \cite[Theorem 2.5]{clarkhazratrigby} it was proved that an ample $\G$-graded groupoid $\mG$ is strongly graded if and only if the Steinberg algebra $A_R(\mG)$ is strongly $\G$-graded. 

\subsection{Skew-product of groupoids}\label{subsection23}

In this subsection we first demonstrate that for the $\G$-graded groupoid $\mG$, graded by the cocycle $c:\mG\xra \G$, the skew-product $\mG \times_c \G$ takes the role of the smash product in the graded ring theory as described in \S\ref{onboardtodenver}. Namely we prove that 
\begin{equation*}
\Gr_R \mG  \cong \Modd_R \mG\times_c \G,
\end{equation*}
where $\Gr_R \mG$ is the category of graded $\mG$-sheaves of $R$-modules with graded morphisms and $\Modd_R \mG$ is the category  of $\mG$-sheaves of $R$-modules (see \S\ref{nesfdte} and \eqref{keymom}). Comparing this with the situation in ring theory given in \eqref{cohenmont}
\begin{equation*}
 \Gr A\longrightarrow \Modd A\#\G,
\end{equation*}
leads us to define the graded homology theory of a $\G$-graded groupiod $\mG$ as the homology theory of skew-product $\mG\times_c\G$ (see Definition~\ref{homology}).

\begin{deff} \label{defsp}
Let $\mG$ be a locally compact Hausdorff groupoid, $\G$ a discrete group and $c : \mG \to \G$ a cocycle. The \emph{skew-product} of $\mG$ by $\G$, denoted by $\mG \times_c \G$,  is the groupoid $\mG
\times \G$ such that $(x,\a)$ and $(y, \b)$ are composable if $x$ and $y$ are
composable and $\a=c(y)\b$. The composition is then given by $$\big (x,
c(y)\b\big)\big(y,\b\big)=(xy, \b)$$ with the inverse $$(x, \a)^{-1}=(x^{-1},c(x)\a).$$ 
\end{deff}

The source map $s:\mG\times_c \G\xra \mG^{(0)}\times \G$ and the range map $r:\mG\times_c \G\xra \mG^{(0)}\times \G$ of the skew-product $\mG\times_c \G$ are given by $s(x, \a)=(s(x), \a)$ and $r(x, \a)=(r(x), c(x)\a)$ for $(x,\a)\in \mG\times_c \G$.

The skew-product $\mG\times_c \G$ is
a $\Gamma$-graded topological groupoid under the product topology on $\mG \times \G$ and with
degree map $\tilde{c}(x,\gamma) := c(x)$.

\begin{rmk} Note that our convention for the composition of the skew-product here is slightly
different from that in \cite[Definition 1.6]{renault} and \cite[\S3.2]{matui}. Since in Section \ref{sectionsix} we will use the representation of the monoid of graded finitely generated projective modules over a Leavitt path algebra which was produced in \cite[\S5]{ahls},  our convention for the composition of the skew-product is the same as that in \cite[Definition 3.2]{ahls}.\end{rmk}

Next we describe the bisections of a skew-product groupoid as they are needed in the sequel. For a $\G$-graded ample groupoid $\mG$, we observe that a graded compact open bisection $U\in B_\a^{\rm co}(\mG\times_c \G)$ with $\a\in\G$ can be written as 
\begin{equation}
\label{normform}
U = \bigsqcup^l_{i=1} U_i \times \{\g_i\},
\end{equation} 
with $\g_i\in\G$ distinct and $U_i$ an $\alpha$-graded compact open bisection of $\mG$. Indeed, if $U \subseteq \tilde{c}^{-1}(\alpha)$, then for
each $\g\in\G$, the set $U \cap \mG \times \{\g\}$ is a compact open bisection. Since
these are mutually disjoint and $U$ is compact, there are finitely many (distinct) $\g_1,
\dots, \g_l \in \G$ such that \[U = \bigsqcup^l_{i=1} U \cap (\mG \times \{\g_i\}).\] Each $U
\cap \mG \times \{\g_i\}$ has the form $U_i \times \{\g_i\}$, where $U_i \subseteq \mG$ is
compact open. The $U_i$ have mutually disjoint sources because the source map on $\mG
\times_c \G$ is just $s \times \operatorname{id}$, and $U$ is a bisection. So each $U_i
\in B^{\rm co}_\alpha(\mG)$, and $U = \bigsqcup^l_{i=1} U_i \times \{\g_i\}$. We observe that \begin{equation*}
s(U)=\bigsqcup_{i=1}^l s(U_i)\times \{\g_i\} \text{~~~~and~~~~} r(U)=\bigsqcup_{i=1}^l r(U_i)\times \{\a\g_i\}.
\end{equation*}

Recall that for a locally compact, Hausdorff groupoid $\mG$, we say that $\mG$ is \emph{effective} if the interior $\Iso(\mG)^\circ$ of $\Iso(\mG)$ is $\mG^{(0)}$. We say that $\mG$ is \emph{principal} if $\Iso(\mG)=\mG^{(0)}$ (\cite{cep}).

\begin{lemma}\label{ghghgrt} Let $\mG$ be a locally compact Hausdorff $\Gamma$-graded groupoid,  with the cocycle $c:\mG\rightarrow \G$.  We have $\Iso(\mG \times_c \G)=\Iso(\mG_\varepsilon)\times \G$. Thus the groupoid  $\mG_\varepsilon$ is effective if and only if $\mG \times_c \G$ is effective. 
Furthermore, $\mG_\varepsilon$ is principal if and only if $\mG \times_c \G$ is principal.
\end{lemma}

\begin{proof}
Let $(x,\alpha) \in \Iso(\mG \times_c \G)$. Then $(s(x),\alpha)=s(x,\alpha)=r(x,\alpha)=(r(x),c(x)\alpha)$. It immediately follows that 
$x\in \Iso(\mG_{\varepsilon})$ and so $\Iso(\mG \times_c \G) \subseteq \Iso(\mG_\varepsilon)\times \G$. The reverse inclusion is immediate. The rest of Lemma follows easily. 
\end{proof}

\subsection{Groupoid spaces and groupoid equivalences}
In this subsection we recall the concepts of imprimitive groupoids, linking groupoids and equivalence spaces between two groupoids. In \S\ref{grgrim} we develop the graded version of these concepts. 

Let $\mG$ be a locally compact groupoid with Hausdorff unit space and $X$ a locally compact space. We say \emph{$\mG$ acts on the left} of $X$ if there is a map $r_X$ from $X$ onto $\mG^{(0)}$ and a map $(\g, x)\mapsto \g\cdot x$ from $$\mG*X:=\{(\g, x)\in \mG\times X\;|\; s(\g)=r_X(x)\}$$ to $X$ such that 

\begin{itemize}
\item[(1)] if $(\eta, x)\in \mG*X$ and $(\g,\eta)$ is composable in $\mG$, then $(\g\eta, x), (\g,\eta\cdot x)\in \mG*X$ and $\g\cdot(\eta\cdot x)=(\g\eta)\cdot x$;

\smallskip

\item[(2)] $r_X(x)\cdot x=x$ for all $x\in X$.
\end{itemize} We call $X$ a \emph{continuous left $\mG$-space} if $r_X$ is an open map and both $r_X$ and $(\g, x)\mapsto \g\cdot x$ are continuous. The action of $\mG$ on $X$ is \emph{free} if $\g\cdot x=x$ implies $\g=r_X(x)$. It is \emph{proper} if the map from $\mG*X\to X\times X$ given by $(\g, x)\mapsto (\g\cdot x, x)$ is a proper map in the sense that inverse images of compact sets are compact. Right actions can be defined similarly, writing $s_X$ for the map from $X$ to $\mG^{(0)}$, and $$X*\mG:=\{(x, \g)\in X\times \mG\;|\; s_X(x)=r(\g)\}.$$

Suppose $X$ is a left $\mG$-space. Define the \emph{orbit equivalence relation} on $X$ determined by $\mG$ to be $x \sim y$ if and only if there exists $\g\in\mG$ such that $\g\cdot x =y$. The quotient space with respect to this relation is denoted $\mG\setminus X$, the elements of $\mG\setminus X$ are denoted $[x]$, and the canonical quotient map is denoted by $q$. When $X$ is a continuous $\mG$-space we will give $\mG\setminus X$ the quotient topology with respect to $q$. When $X$ is a right $\mG$-space the orbit equivalence relation is defined similarly and we will use exactly the same notation.
In some cases $X$ will be both a left $\mG$-space and a right $\HH$-space, where $\mG$ and $\HH$ are groupoids, and in this situation we will denote the orbit space with respect to the $\mG$-action by $\mG\setminus X$ and the orbit space with respect to the $\HH$-action by $X/\HH$.

\begin{deff} Let $\mG$ and $\HH$ be locally compact groupoids. A locally compact space $Z$ is called a \emph{$(\mG, \HH)$-equivalence} if 

\begin{enumerate}[\upshape(1)]
\item $Z$ is a free and proper left $\mG$-space;

\smallskip

\item $Z$ is a free and proper right $\HH$-space;

\smallskip

\item the actions of $\mG$ and $\HH$ on $Z$ commute;

\smallskip

\item $r_Z$ induces a homeomorphism of $Z/\HH$ onto $\mG^{(0)}$;

\smallskip

\item $s_Z$ induces a homeomorphism of $\mG\setminus Z$ onto $\HH^{(0)}$.
\end{enumerate}
\end{deff}

 Let $\HH$ be a locally compact Hausdorff groupoid, and let $Z$ be a free and proper right $\HH$-space. Then $\HH$ acts diagonally on 
\begin{equation}
\label{defzz}
Z *_s Z=\{(x, y)\in Z\times Z\;|\; s(x)=s(y)\},
\end{equation}
and we define $Z^{\HH}= (Z *_s Z)/\HH$. Two equivalence classes $[z,w]_{\HH},[x,y]_{\HH}\in Z^{\HH }$ are composable when $[w] = [x]$ in $Z/\HH$, and we define
$[z,w]_{\HH} [w,y]_{\HH}= [z,y]_{\HH}$ and $[z,w]_{\HH}^{-1} = [w,z]_{\HH}.$
It can be checked that these operations make $Z^{\HH}$ into a locally compact Hausdorff groupoid (see \cite[Proposition 1.92]{goehle}). Furthermore, the range and source maps are given by
$r([z,w]_{\HH}) = [z,z]_{\HH}$ and $s([z,w]_{\HH}) = [w,w]_{\HH}$,
which allows us to identify the unit space of $Z^{\HH}$ with $Z/\HH$. The groupoid $Z^{\HH}$ is called the \emph{imprimitivity groupoid} associated to the $\HH$-space $Z$. The imprimitivity groupoid associated to a left action is defined analogously. The groupoid $Z^{\HH}$ admits a natural left action on $Z$, which makes $Z$ into a $(Z^{\HH}, \HH)$-equivalence (see \cite[Proposition 1.93]{goehle}). Thus any free and proper groupoid space gives rise to a groupoid equivalence in a canonical way.  Indeed, this construction is prototypical: if $Z$ is a $(\mG,\HH)$-equivalence, then $Z^{\HH}$ and $\mG$ are naturally isomorphic via the map
\begin{equation}
\label{isogroupoid}
[z, w]_{\HH} \longmapsto _{\mG}[z, w],
\end{equation} 
where $_{\mG}[z, w] \in \mG$ is the unique element satisfying $_{\mG}[z, w] \cdot w = z$ (see \cite[Proposition 1.94]{goehle}).

Given a $(\mG, \HH)$-equivalence $Z$, define the opposite space $Z^{\rm op}=\{\overline{z} : z\in Z\}$ to be a homeomorphic copy of $Z$, but let $\mG$ and $\HH$ act on the right and left of $Z^{\rm op}$ as follows:
   $$r(\overline{z})=s(z), s(\overline{z})=r(z), \eta\cdot \overline{z}=\overline{z\cdot \eta^{-1}}, \overline{z}\cdot y=\overline{y^{-1}\cdot z}$$
for $\eta \in H, y\in\mG$. It is straightforward to check that this makes $Z^{\rm op}$ into an $(\HH, \mG)$-equivalence. We then define the \emph{linking groupoid} to be the topological disjoint union $$L= \mG \sqcup Z \sqcup Z^{\rm op} \sqcup \HH $$ with 
$r,s: L\xra L^{(0)}: =\mG^{(0)} \sqcup \HH^{(0)}$ inherited from the range and source maps on each of $\mG, \HH, Z$ and $Z^{\rm op}$.
The set of composable pairs is
$L^{(2)} =\{(x,y)\in L\times L: s(x)=r(y)\}$ and multiplication $(k, l)\mapsto kl$ in $L$ is given by
\begin{itemize}
\item multiplication in $\mG$ and $\HH$ when $(k, l)$ is a composable pair in $\mG$ or $\HH$;

\smallskip

\item $kl=k\cdot l$ when $(k, l)\in (Z* H)\sqcup (\mG*Z)\sqcup (H*Z^{\rm op})\sqcup (Z^{\rm op}*\mG)$;

\smallskip

\item $kl=_{\mG}[k, h]$ if $k\in Z$ and $l=\overline{h}\in Z^{\rm op}$, and $kl=[h, l]_{\HH}$ if $j\in Z$ and $k=\overline{h}\in Z^{\rm op}$.
\end{itemize} The inverse map is the usual inverse map in each of $\mG$ and $\HH$ and is given by $z\mapsto \overline{z}$ on $Z$ and $\overline{z}\mapsto z$ on $Z^{\rm op}$. It is shown in \cite[Lemma 2.1]{sw} that these operations make $L$ into a locally compact Hausdorff groupoid.

\subsection{Graded groupoid equivalences}\label{grgrim}
Let $\G$ be a group and let $\mG$ be a $\G$-graded groupoid, graded by the cocycle $c: \mG\xra \G$. Recall from \cite[\S3.4]{clarkhazratrigby} that a $\mG$-space $X$ is called graded $\mG$-space if there is a continuous map $k: X\xra \G$ such that $k(\g\cdot x)=c(\g)k(x)$, whenever $(\g,x)\in \mG*X$.

Let $\mG$ and  $\HH$ be $\G$-graded groupoids. A $(\mG,\HH)$-equivalence $Z$ is called a \emph{$\G$-graded $(\mG,\HH)$-equivalence} if $Z$ is a $\G$-graded $\mG$-space and also a $\G$-graded $\HH$-space with respect to a continuous map $k: Z\xra \G$.

Let $\HH$ be a $\G$-graded groupoid graded by the cocycle $c: \HH\xra \G$ and $X$ a $\G$-graded right $\HH$-space with $k:X\xra \G$ the grading map. Then the imprimitivity groupoid $X^{\HH}$ is a $\G$-graded groupoid graded by the cocycle $k: X^{\HH}\xra \G$ given by 
$$k([x, y])=k(x)k(y)^{-1}$$ for $[x,y]\in X^{\HH}$. The map $k: X^{\HH}\xra \G$ is well defined. Observe that for $[x',y']=[x,y]$ in $X^{\HH}$ there exists $\g\in\HH$ such that $x'=x\cdot \g$ and $y'=y\cdot \g$, and thus $k([x',y'])=k(x')k(y')^{-1}=k(x\cdot \g)k(y\cdot \g)=k([x, y])$. We observe that $k$ is a continuous map. Indeed, we have the commutative diagram 
\begin{displaymath}
\xymatrix{X\times X\ar[r]^<<<<<{\phi}\ar[d]_{q}& \G \\X *_sX\ar[ur]_<<<<<<<<{\psi}}
\end{displaymath} where $\phi(x, y)=k(x)k(y)^{-1}$ for $(x, y)\in X\times X$, $ X *_sX$ is defined in \eqref{defzz}, $q: X\times X\xra X *_sX$ is the quotient map, and $\psi(x, y)=k(x)k(y)^{-1}$. Observe that $\phi$ is continuous and that $\phi^{-1}(\g)=q^{-1}(\psi^{-1}(\g))$ for any $\g\in\G$. Thus $\psi^{-1}(\g)$ is open, since $\phi^{-1}(\g)$ is open and $q$ is the quotient map. Hence $\psi$ is continuous. Similarly, we have the following commutative diagram \begin{displaymath}
\xymatrix{X*_sX\ar[r]^<<<<<{\psi}\ar[d]_{q}& \G \\X^{\HH}\ar[ur]_<<<<<<<<{k}}
\end{displaymath} and thus $k: X^{\HH}\xra \G$ is continuous. In the case that $Z$ is a $\G$-graded $(\mG, \HH)$-equivalence, the groupoid isomorphism given in \eqref{isogroupoid} preserves the grading and thus $\mG$ and $Z^{\HH}$ are isomorphic as $\G$-graded groupoids.

If $Z$ is a $\G$-graded $(\mG,\HH)$-equivalence, then the linking groupoid $L$ of $Z$ is $\G$-graded with the cocycle map $k: L\xra \G$ inherited from the grading maps of $\mG$, $\HH$  and $Z$, and $k(\overline{z})=k(z)^{-1}$ for $z\in Z$.

\subsection{The category of graded $\mG$-sheaves} \label{nesfdte} Let $\mG$ be a $\Gamma$-graded groupoid, graded by the cocycle $c:~\mG~\rightarrow~\Gamma$
and 
$E$ a right $\mG$-space.  For $x\in \mG^0$, the fibre $s_E^{-1}(x)$ of $E$ is denoted by $E_x$ and is called the \emph{stalk} of $E$ at $x$. 

A $\mG$-space $E$ is called a \emph{$\mG$-sheaf of sets} if $s_E\colon E\to \mG^ {(0)}$ is a local homeomorphism. For a commutative ring $R$ with unit, we say the sheaf $E$ is a \emph{graded $\mG$-sheaf of $R$-modules}  if 
each stalk $E_x$ is a (left) $R$-module such that 

\begin{enumerate}[\rm(1)]

	\item the zero section sending $x\in \mG^ {(0)}$ to the zero of $E_x$ is continuous;
	
	\smallskip

	\item addition $E\times_{\G^{( 0)}} E\rightarrow E$ is continuous;
	
	\smallskip

	\item scalar multiplication $R\times E\rightarrow E$ is continuous, where $R$ has the discrete topology;
	
	\smallskip

	\item for each $g\in \mG$, the map $\psi_g\colon E_{r(g)}\rightarrow E_{s(g)}$ given by $\psi_g(e) = eg$ is $R$-linear.

	\item for any $x\in \mG^{(0)}$, $E_x=\bigoplus_{\gamma\in \Gamma} 
	(E_x)_\gamma$, where $(E_x)_\gamma$ are $R$-submodules of $E_x$;	
	\smallskip
	
	\item  $E_\gamma := \bigcup_{x \in \mG^{(0)}} (E_x)_\gamma$ is open in $E$ for every $\gamma \in \Gamma$; and
	\smallskip
	\item $E_\gamma \mG_\delta\subseteq E_{\gamma\delta}$ for every $\gamma, \delta \in \Gamma$.
\end{enumerate}
We call $(E_x)_\gamma$ the $\gamma$-\textit{homogeneous component} of $E_x$, and denote the \textit{homogeneous elements} of $E$ by 
\[E^h:=\bigcup_{x \in \mG^{(0)},\gamma \in \Gamma} (E_x)_\gamma.\] Note that the map $k: E^h \to \Gamma$, $s \mapsto \gamma$, where $ s\in (E_x)_\gamma$, is continuous, and (7) can be interpreted as $k(eg) = k(e)c(g)$ for every homogeneous $e \in E^h$ and any $g \in \mG$ such that $s_E(e) = r(g)$. Unlike graded $\mG$-spaces, it does not make sense to define a degree $k$ on all of $E$. A morphism $\phi:E\to F$ of $\mG$-sheaves of $R$-modules is a \emph{graded morphism} if $\phi(E_\g)\subseteq F_\g$ for any $\g \in \Gamma$. The category of graded $\mG$-sheaves of $R$-modules with graded morphisms will be denoted $\Gr_R\mG$. If from the outset we set the grade group $\Gamma$ the trivial group, then our construction gives the category of $\mG$-sheaves of $R$-modules, denoted by $\Modd_R \mG$.

\subsection{Graded Kakutani equivalences and stable groupoids}
\label{subsection38}

Matui \cite{matui} defines Kakutani equivalence for ample groupoids with compact unit spaces. In \cite{carlsenruizsims}, the authors extend this notion to ample groupoids with non-compact unit spaces. In this subsection, we consider graded groupoids and graded Kakutani equivalences.  This allows us to show that for a $\G$-graded ample groupoid $\mG$, 
\begin{equation*}
\Gr_R \mG  \cong \Modd_R \mG\times_c \G,
\end{equation*}
(see \eqref{keymom} and compare this with \eqref{cohenmont}).
This justifies the use of skew-product $\mG\times_c \G$ for the definition of the graded homology of the $\G$-graded groupoid $\mG$ (see Definition~\ref{homology}).

Recall that two ample groupoids $\mG$ and $\HH$ are \emph{Kakutani equivalent} if there are a $\mG$-full clopen $X\sub\mG^{(0)}$ and an $\HH$-full clopen $Y\sub\HH^{(0)}$ such that $\mG|_X\cong 
\HH|_Y$.

\begin{deff}\label{gradedkak}
Let $\G$ be a discrete group. Suppose that $\mG_1$ and $\mG_2$ are $\G$-graded groupoids. The groupoids $\mG_1$ and $\mG_2$ are called \emph{$\G$-graded Kakutani equivalent } if there are full clopen subsets $Y_i\subseteq \mG_i^{0}$, for $i=1,2$, such that $\mG_1|_{Y_1}$ is graded isomorphic to $\mG_2|_{Y_2}$. 
\end{deff}

Let $\mG$ be a $\G$-graded groupoid graded by the cocycle $c:\mG \rightarrow \G$. The \emph{stable groupoid} $\mG_\Gamma$ is defined as 
\begin{equation}\label{stagrome}
\mG_\Gamma=\big \{g_{\alpha,\beta} \mid \alpha,\beta\in \G, g\in \mG\big \},
\end{equation}
with composition given by $g_{\alpha,\beta} h_{\beta,\gamma}=(gh)_{\alpha,\gamma}$, if $g$ and $h$ are composable and inverse given by $(g_{\alpha, \beta})^{-1}=g^{-1}_{\beta, \alpha}$. Here $\mG_\Gamma^{(0)}=\{x_{\alpha,\alpha}\mid \alpha \in \G, x\in \mG^{(0)}\}$, $s(g_{\alpha,\beta})=s(g)_{\beta,\beta}$ and $r(g_{\alpha,\beta})=r(g)_{\alpha,\alpha}$. Observe that if $\mG$ is an ample groupoid, so is $\mG_\G$.

One defines a cocycle  (called $c$ again) 
\begin{align*}
c: \mG_\G&\longrightarrow \Gamma,\\
g_{\alpha,\beta} &\longmapsto \alpha^{-1} c(g)\beta.\notag
\end{align*}
 This makes  $\mG_\G$ a $\G$-graded groupoid. Furthermore, the $\varepsilon$-component of $\mG_\G$,
$(\mG_\G)_\varepsilon=\{ g_{\alpha,\beta} \mid c(g)=\alpha\beta^{-1} \}$,  is $\G$-graded groupoid by the cocyle
\begin{align*}
c: (\mG_\G)_\varepsilon&\longrightarrow \Gamma,\\
g_{\alpha,\beta} &\longmapsto c(g).\notag
\end{align*}
 
 Recall that for a $\G$-graded \'etale groupoid $\mG$, the skew-product $\mG\times_c \G$ given in Definition \ref{defsp} is $\Gamma$-graded with components 
 $(\mG\times_c \G)_\gamma=\mG_\gamma\times \G$, $\gamma \in \Gamma$ (see \cite[\S 3.3]{ahls} and  \cite[Definition~1.4]{renault}). 
 
\begin{lemma}\label{prinp}
Let $\mG$ be a $\G$-graded groupoid. We have the following. 

\begin{enumerate}[\upshape(1)]

\item If $\Gamma$ is a trivial group then $\mG_\G=\mG.$

\smallskip

\item  The groupoid  $\mG_\G$ is strongly $\G$-graded. 

\smallskip 

\item There is a $\G$-graded isomorphism $\mG \times_c \G \cong (\mG_\G)_\varepsilon$. 
\smallskip

\item The groupoids $\mG$ and $\mG_\G$ are graded Kakutani equivalent.  
\end{enumerate}

\end{lemma}

\begin{proof}

(1) This follows from the construction~(\ref{stagrome}).  

(2) A $\G$-groupoid $\mG$ is strongly graded if and only if $s(\mG_\gamma)=\mG^{(0)}$ for every $\gamma \in \Gamma$ (see \cite[Lemma 2.3]{clarkhazratrigby}). 
Let $x_{\alpha,\alpha} \in \mG_\Gamma^{(0)}$, where $\alpha \in \Gamma$ and $x\in \mG^{(0)}$. Then $x_{\alpha\g,\alpha} \in (\mG_\Gamma)_{\gamma^{-1}}$, with $s(x_{\a \g,\alpha})=x_{\alpha,\alpha}$. Thus $\mG_\G$ is strongly graded.

(3) Consider the map 
\begin{align}
\label{stableiso}
\phi: \mG\times_c \G  &\longrightarrow  (\mG_\G)_\varepsilon,\\
(g,\gamma) &\longmapsto g_{c(g) \g,\gamma}. \notag
\end{align}
It is easy to show that this map is a $\G$-graded groupoid isomorphism. 

(4) Note that if $Y$ is a full clopen subset of $\mG^{(0)}$ if and only if  $Y_\varepsilon:=\{x_{\varepsilon,\varepsilon} \mid x \in \ Y\}$ is a full clopen subset in $\mG_\G$. The map 
\begin{align}
\label{griso}
\psi: \mG|_{Y} & \longrightarrow \mG_\G|_{Y_\varepsilon},\\
g &\longmapsto g_{\varepsilon,\varepsilon},  \notag
\end{align}
is a $\Gamma$-graded groupoid isomorphism. Now considering $Y=\mG^{(0)}$ the claim follows. 
\end{proof}

In~\cite[Theorem~2.12]{cm1984} it was shown that for a $\G$-graded ring $A$, where $\G$ is a finite group, $A$ is strongly graded if and only if $A_{\varepsilon}$ and $A \#\G$ are Morita equivalent. We have a similar statement in the groupoid setting, confirming skew-product groupoids are the right replacement for smash products.

\begin{thm} \label{tfse}Let $\mG$ be a $\G$-graded \'etale groupoid. The following statements are equivalent:
\begin{enumerate}[\upshape(1)]
\item The groupoid $\mG$ is strongly graded;

\smallskip

\item $\mG_{\varepsilon}$ and $\mG\times_c \G$ are Kakatuni equivalent with respect to some full clopen subset $Y\subseteq \mG_{\varepsilon}^{(0)}$ in $\mG_{\varepsilon}$ and $Y_{\varepsilon}=\{(y, \varepsilon)\;|\; y\in Y\}$ a full clopen subset of $\mG^{(0)}\times \G$ in $\mG\times_c \G$.
\end{enumerate}
\end{thm}

\begin{proof} (1)$\Rightarrow$ (2): Suppose that $\mG$ is strongly graded. We first claim that $Y_{\varepsilon}=\{(y, \varepsilon)\;|\; y\in Y\}$ is a full clopen subset of $\mG^{(0)}\times \G$ in $\mG\times_c \G$, when $Y\subseteq \mG_{\varepsilon}^{(0)}$ is clopen full in $\mG_{\varepsilon}$. Take any $(y, \g)\in \mG^{(0)}\times \G$. For $y\in \mG^{(0)}$, there exists $g\in\mG_{\g}$ such that $r(g)=y$, since $\mG$ is strongly graded. For $s(g)\in \mG^{(0)}$, there exists $g'\in \mG_{\varepsilon}$ such that $s(g')\in Y$ and $r(g')=s(g)$, since $Y\sub\mG^{(0)}$ is full in $\mG_{\varepsilon}$. Then we have $(gg', \varepsilon)\in \mG\times_c\G$ such that $r(gg', \varepsilon)=(y, \g)$ and $s(gg', \varepsilon)=(s(g'), \varepsilon)\in Y_{\varepsilon}$. Since $\mG\times_c\G$ has product topology, $Y_{\varepsilon}$ is clopen in $\mG^{(0)}\times \G$ when $Y$ is clopen in $\mG^{(0)}$. Thus we complete the proof of the claim. Observe that $(\mG|_Y)_{\varepsilon}=\mG_{\varepsilon}|_Y$ and $(\mG_{\G}|_{Y_{\varepsilon}})_{\varepsilon}=(\mG_{\G})_{\varepsilon}|_{Y_{\varepsilon}}$. By \eqref{griso} and \eqref{stableiso} , we have $(\mG_{\varepsilon})|_{Y}\cong (\mG\times_c\G)|_{Y_{\varepsilon}}$ as groupoids. The unit space $\mG^{(0)}$ is a full clopen subset for $\mG_{\varepsilon}$. Hence, (ii) holds.

(2)$\Rightarrow$(1): Suppose that (ii) holds. Then there exist $Y\subseteq \mG_{\varepsilon}^{(0)}$ a full clopen in $\mG_{\varepsilon}$ and $Y_{\varepsilon}=\{(y, \varepsilon)\;|\; y\in Y\}$ a full clopen subset of $\mG^{(0)}\times \G$ in $\mG\times_c \G$. We need to show that $r(\mG_{\g})=\mG^{(0)}$ for any $\g\in\G$. Take any $\g\in\G$ and any $u\in\mG^{(0)}$. There exists $(g, \eta)\in \mG\times_c\G$ such that $s(g, \eta)\in Y_{\varepsilon}$ and $r(g, \eta)=(u, \g)$. This implies $\eta=\varepsilon$ and $g\in\mG_{\g}$. Thus we have $r(\mG_{\g})=\mG^{(0)}$ for any $\g\in\G$.
\end{proof}

\section{Steinberg algebras}
\label{subsection24}

In this section we briefly recall the construction of the Steinberg algebra associated to an ample groupoid. We will then tie together the concepts between graded groupoid theory and graded ring theory via Steinberg theory. 

Let $\mG$ be an ample topological groupoid. Suppose that $R$ is a commutative ring with unit. Consider $A_R(\mG) := C_c(\mG, R)$, the space of
compactly supported continuous functions from $\mG$ to $R$ with $R$ given the discrete topology.
Then $A_R(\mG)$ is an $R$-algebra with addition defined point-wise and multiplication by 
$$(f*g)(\g) = \sum_{\a\b=\g}f(\a)g(\b).$$
It is useful to note that $$1_{U}*1_{V}=1_{UV}$$ for compact open bisections $U$ and $V$
(see \cite[Proposition 4.5(3)]{st}). With this structure, $A_R(\mG)$ is an algebra called the \emph{Steinberg algebra}  associated to $\mG$. The algebra $A_R(\mG)$ can also be realised as the span of characteristic functions of the form $1_U$, where $U$ is a compact open bisection. By \cite[Lemma 2.2]{cm} and \cite[Lemma
3.5]{cfst}, every element $f\in A_{R}(\mG)$ can be expressed as
\begin{equation}
\label{express}
f=\sum_{U\in F}a_{U}1_{U},
\end{equation}
where $F$ is a finite subset of mutually disjoint compact open bisections.

Recall from \cite[Lemma 3.1]{cs} that if $\mG=\bigsqcup_{\g\in \G}\mG_\g$ is a $\G$-graded ample groupoid, then the Steinberg algebra $A_R(\mG)$ is a $\G$-graded
algebra with homogeneous components
\begin{equation*}\label{hgboat}
    A_{R}(\mG)_{\g} = \{f\in A_{R}(\mG)\mid \supp(f)\subseteq \mG_\g\}.
\end{equation*}
The family of all idempotent elements of $A_{R}(\mG^{(0)})$ is a set of local units for
$A_{R}(\mG)$ (see \cite[Lemma 2.6]{cep}). Here, $A_{R}(\mG^{(0)})\subseteq A_{R}(\mG)$ is a
subalgebra. Note that any ample 
groupoid admits the trivial cocycle from $\mG$ to the trivial group $\{\varepsilon\}$,
which gives rise to a trivial grading on $A_{R}(\mG)$.

 Let $A$ be an $R$-algebra. A representation of $B^{\rm co}_*(\mG)$ in $A$ is a family 
$$\big \{t_U \mid U\in B^{\rm co}_*(\mG)\big \}\sub A$$ satisfying
\begin{itemize}
\item[(R1)] $t_{\emptyset}= 0$;

\smallskip

\item[(R2)] $t_Ut_V = t_{UV}$ for all $U,V\in B^{\rm co}_*(\mG)$, and

\smallskip
\item[(R3)] $t_U +t_V = t_{U\cup V}$, whenever $U,V \in B^{\rm co}_\g(\mG)$ for some $\g\in\G$, $U\cap V =\emptyset$ and $U\cup V\in B^{\rm co}_\g(\mG)$.
\end{itemize}

We have the following statement that $A_R(\mG)$ is a universal algebra (see \cite[Theorem 3.10]{cfst} and \cite[Proposition 2.3]{cm}), which will be used throughout the paper. 

\begin{lemma}\label{copenhagenairport} 
Let $\mG$ be a $\G$-graded ample groupoid. Then 
$\big \{1_U \mid U\in B^{\rm co}_*(\mG)\big \}\sub A_R(\mG)$ is a representation of 
$B^{\rm co}_*(\mG)$ which spans $A_R(\mG)$. Moreover, $A_R(\mG)$ is universal for representations of 
$B^{\rm co}_*(\mG)$ in the
sense that for every representation 
$\big \{t_U \mid U\in B^{\rm co}_*(\mG)\big \}$ of $
B^{\rm co}_*(\mG)$ in an $R$-algebra $A$, there
is a unique $R$-algebra homomorphism $\pi: A_R(\mG)\xra A$ such that $\pi(1_U) = t_U$, for all $U\in B^{\rm co}_*(\mG)$.
\end{lemma}

\subsection{The Steinberg algebra of a semi-direct product groupoid} \label{intheairfaroe}

Let $\mG$ be a $\Gamma$-graded ample groupoid, graded by the cocycle $c:\mG\xra \G$, $R$ a unital commutative
ring and $\mG \times_c \Gamma$ the skew product groupoid (see Definition~\ref{defsp}). By~\cite[Theorem 3.4]{ahls} there is an isomorphism of $\G$-graded algebras 
\begin{align}\label{saiso}
A_{R}(\mG\times_c\G) &\longrightarrow  A_{R}(\mG)\#\G,\\
1_{U} &\longmapsto \sum_{i=1}^l1_{U_i} p_{\g_i}, \notag
\end{align} 
where $U=  \bigsqcup^l_{i=1} U_i \times \{\g_i\}  \in B_*^{\rm co}(\mG\times_c \G)$ (see~(\ref{normform})). This will be used in the proof of Theorem~\ref{ses} to calculate the graded homology of graph groupoids.


In this subsection we establish the dual of this statement. We show that the Steinberg algebra of the semi-direct product groupoid of an ample groupoid (Example~\ref{exmsemidirect}) is graded isomorphic to the partial skew group ring of its Steinberg algebra (Proposition~\ref{jdjdjdue3}).

Recall from \cite[\S 3.1]{hl} the partial action $\phi=(\phi_\g, X_\g, X)_{\g \in \G}$ of a discrete group $\G$ on a locally compact Hausdorff topological space $X$. In the case $X$ is an algebra or a ring then the subsets $X_\g$ should also be ideals and the maps $\phi_\g$ should be isomorphisms of algebras. 
There is a $\G$-graded groupoid \begin{equation*}
 \label{groupoid}
 \mG_X=\bigcup_{\g \in \G} \g \times X_\g,
 \end{equation*} associated to a partial action $(\phi_\g, X_\g, X)_{\g \in \G}$, whose composition and inverse maps are given by $(\g,x)(\g',y)=(\g\g',x)$ if $y=\phi_{\g^{-1}}(x)$ and $(\g,x)^{-1}=(\g^{-1},\phi_{\g^{-1}}(x))$.

The action $\phi:\Gamma\curvearrowright \mG$ is a partial action of $\G$ on $\mG$ such that $\phi_\g$ is a bijection from $\mG$ to $\mG$. Thus the action $\phi$ induces an action of $\G$ on $A_R(\mG)$, still denoted by $\phi$, such that 
\begin{equation}
\label{inducedaction}
\phi_\g(f)=f\circ \phi_{\g^{-1}}
\end{equation} for $\g\in\G$ and $f\in A_R(\mG)$ (see \cite[\S 3.1]{hl}).

For an action $\phi:\Gamma\curvearrowright \mG$ with $\mG=\mG^{(0)}$, observe that the semi-direct product coincides with the groupoid $\mG_{X}=\bigcup_{\g \in \G} \g \times X$ with $X=\mG^{(0)}$. 

Let $\phi=(\phi_\g, A_\g, A)_{\g\in\G}$ be a partial action of the discrete group $\G$ on an algebra $A$. The \emph{partial skew group ring} $A\rtimes_{\phi}\G$ consists of all formal forms $\sum_{\g\in \G}a_\g\d_\g$ (with finitely many $a_\g$ nonzero), where $a_\g\in A_\g$ and $\d_\g$ are symbols, with addition defined in the obvious way and multiplication being the linear extension of $$(a_\g\d_\g) (a_{\g'}\d_{\g'})=\phi_\g\big(\phi_{\g^{-1}}(a_\g)a_{\g'}\big)\d_{\g\g'}.$$ Observe that $A \rtimes_{\phi}\G$ is always a $\G$-graded ring with $(A \rtimes_{\phi}\G)_\g=A_\g  \d_\g$ for $\g\in\G$.

Recall from \cite[Proposition 3.7]{hl} that the Steinberg algebra of $\mG_X$ is the partial skew group ring $C_R(X)\rtimes_{\phi}\G$. We show here that the Steinberg algebra of the semi-direct product $\mG\rtimes_{\phi}\G$ is the partial skew group ring $A_R(\mG)\rtimes_{\phi}\G$.

\begin{prop}\label{jdjdjdue3}
Let $\mG$ be an ample groupoid and $\phi:\Gamma\curvearrowright \mG$ an action of a discrete group $\G$ on $\mG$ and $R$ a commutative ring with unit. Then there is a $\G$-graded $R$-algebra isomorphism $$A_R(\mG\rtimes_{\phi}\G)\cong A_R(\mG)\rtimes_{\phi}\G.$$ 
\end{prop}
\begin{proof} We first define a representation $\{t_{U} \mid  U\in B_{*}^{\rm co}(\mG\rtimes_{\phi} \G)\}$ in the algebra $A_{R}(\mG)\rtimes_{\phi}\G$. For any graded compact open bisection $U\in B^{\rm co}_{\a}(\mG\rtimes_{\phi}\G)$ of $\mG\rtimes_{\phi}\G$, we may write $U=U\cap \mG\times\{\a\}=U_{\a}\times\a$, where $U_{\a}\sub \mG$ is a compact open bisection of $\mG$. We define $t_U=1_{U_{\a}}\delta_{\a}$. We show that these elements $t_U$ satisfy (R1)--(R3). Certainly if $U=\emptyset$, then
$t_{U}=0$, giving~(R1). For~(R2), take $V \in B^{\rm co}_\b(\mG \rtimes_{\phi} \G)$, and
write $V = V_{\b}\times \{\b\}$. Then
\begin{equation} \label{tt}
\begin{split}
t_{U}t_{V}
    &=1_{U_{\a}}\delta_{\a}\cdot1_{V_{\b}}\delta_{\b}\\
    &=\phi_{\a}(\phi_{\a^{-1}}(1_{U_\a})1_{V_{\b}})\delta_{\a\b}\\
    &=1_{U_\a}*\phi_{\a}(1_{V_\b})\delta_{\a\b}\\
    &=1_{U_\a}*1_{\phi_\a(V_\b)}\delta_{\a\b}\\
    &=1_{U_\a\phi_\a(V_\b)}\delta_{\a\b}.
\end{split}
\end{equation} Here, the second last equality holds since $\phi_{\a}(1_{V_\b})=1_{\phi_{\a}(V_\b)}$ by \eqref{inducedaction}. On the other hand, by the composition of the semi-direct product $\mG\rtimes_{\phi} \Gamma$, we have
\begin{equation*}
\begin{split}
UV  &=U_{\a}\times\{\a\}\cdot V_{\b}\times\{\b\}=U_\a\phi_\a(V_\b)\times\{\a\b\}.
\end{split}
\end{equation*} Observe that $U_\a \phi_\a(V_\b)$ is compact since $U_\a$ and $\phi_\a(V_\b)$ are compact. It follows that $t_{UV}=1_{U_a \phi_\a(V_\b)}\delta_{\a\b}$. Comparing with \eqref{tt}, we have $t_{U}t_V=t_{UV}$. For (R3), suppose that $U$ and $V$ are disjoint elements of $B^{\rm
co}_{\g}(\mG\rtimes_{\phi} \G)$ for some $\g\in \G$ such that $U\cup V$ is a bisection
of $\mG\rtimes_{\phi} \Gamma$. Write them as $U=U_{\g}\times \{\g\}$ and $V=V_{\g}\times \{\g\}$ such that $U_\g$ and $V_{\g}$ are disjoint. 
We have $t_{U\cup V}=1_{U_\g\cup V_\g}\delta_\g=(1_{U_\g}+1_{V_\g})\delta_\g=t_U+t_V$.

By the universality of Steinberg algebras, we have an $R$-homomorphism,
\begin{equation*}
\Phi:A_{R}(\mG\rtimes_{\phi}\G) \longrightarrow A_{R}(\mG)\rtimes_{\phi}\G
\end{equation*}
such that $\Phi(1_{B\times \{\a\}})=1_{B}\delta_{\a}$ for each compact open bisection $B$ of $\mG$ and $\a\in\G$.  From the
definition of $\Phi$, it is evident that $\Phi$ preserves the grading. Hence, $\Phi$ is a
homomorphism of $\G$-graded algebras.

It remains to show that $\Phi$ is an isomorphism. For injectivity, take $f\in A_R(\mG\rtimes_{\phi}\G)_\g$ with $\g\in\G$ such that $\Phi(f)=0$. We may write $f=\sum_{i=1}^nr_i1_{U_i}$ with $U_i\in B^{\rm co}_{\g}(\mG\rtimes_{\phi}\G)$ such that $U_i$'s are mutually disjoint. Each $U_i=W_i\times\{\g\}$ with $W_i\sub \mG$ is a compact open bisection of $\mG$ and $W_i$'s are mutually disjoint. Then we have $\Phi(f)=\sum_{i=1}^nr_i1_{W_i}\delta_{\g}=0$, and thus $\sum_{i=1}^nr_i1_{W_i}=0$, implying that each $r_i$ is zero. Hence $f=0$. For the surjectivity of $\Phi$, take $f'\delta_\g\in A_R(\mG)\rtimes_{\phi}\G$ with $\g\in\G$ such that $f'=\sum_{j=1}^mt_j1_{W_j}$ with $t_j\in R$ and $W_j$ disjoint compact open bisections of $\mG$. We have $\Phi(\sum_{j=1}^mt_j1_{W_j\times \{\g\}})=f'$. 
\end{proof}


 \subsection{Steinberg algebras of stable groupoids and graded matrix groupoids}

Next we relate the notion of graded matrix rings (\S\ref{frankairport}) with stable groupoids (\S\ref{subsection38}) and graded matrix groupoids (\S\ref{subsection22}) via Steinberg algebras. 
 
 
 \begin{prop}\label{algebrafor}
 Let $\G$ be a discrete group, $\mG$ a $\G$-graded ample groupoid and $R$ a commutative ring with unit. We have $A_R(\mG_\G)\cong \M_{\G}(A_R(\mG))(\G)$ as $\G$-graded $R$-algebras. 
 \end{prop}
 \begin{proof}We first define a representation $\{t_{U} \mid  U\in B_{*}^{\rm co}(\mG_\G)\}$ in
the algebra $\M_{\G}(A_{R}(\mG))(\G)$. If $U$ is a $\g$-graded compact open bisection of $\mG_\G$, then for $\a, \b\in\G$, the set $U \cap \mG_{\a\b}$ is a compact open bisection, where $\mG_{\a\b}=\{g_{\a\b}\;|\; g\in\mG\}$. Since these are mutually disjoint and $U$ is compact, there are finitely many (distinct) pairs $(\a, \b)$ such that 
\begin{equation*}
\label{union}
U = \bigsqcup_{(\a,\b)\in\G\times\G} U \cap \mG_{\a\b}.
\end{equation*} Each $U
\cap \mG_{\a\b}$ has the form $U_{\a\b}=\{g_{\a\b}\;|\; g\in U'_{\a\b}\}$ for $U'_{\a\b} \subseteq \mG$ a compact open subset. Observe that $U'_{\a\b}$ is a bisection, since $U=\bigsqcup_{\a,\b\in\G} U_{\a\b}$ is a bisection. Take any $x_{\a\b}\in U_{\a\b}$. We have $c(x_{\a\b})=\a^{-1}c(x)\b=\g$, implying $c(x)=\a\g\b^{-1}$ for all $x\in U'_{\a\b}$. So each $U'_{\a\b}
\in B^{\rm co}_{\a\g\b^{-1}}(\mG)$, and $U = \bigsqcup_{(\a,\b)\in\G\times\G} U_{\a\b}$. Using this
decomposition, we define
\begin{equation}
\label{deffor}
    t_{U}=\sum_{(\a,\b)\in\G\times\G} A_{\a\b}, 
\end{equation} where $A_{\a\b}$ is the matrix with $1_{U'_{\a\b}}$ in the $(\a,\b)$-th position and all the other entries zero. Observe that each $A_{\a\b}$ is in $\M_{\G}(A_R(\mG))_\g$.

Obviously, $t_{\emptyset}=0$ and thus (R1) holds. For (R2), take $U, V\in B^{\rm co}_*(\mG_\G)$. We assume that $U=\bigsqcup_{(\a,\b)}U_{\a\b}$ for finitely many distinct pairs $(\a,\b)\in\G\times\G$ and $V=\bigsqcup_{(\mu,\nu)}V_{\mu\nu}$ finitely many distinct pairs $(\mu,\nu)\in\G\times\G$. On one hand, we have \begin{equation}
\label{multi}
\begin{split}
t_Ut_V
&=\sum_{(\a, \b)}\sum_{(\mu, \nu)} A_{\a\b}A_{\mu\nu}\\
&=\sum_{(\a, \b)}\sum_{(\mu, \nu)} A_{\a\b}A_{\mu\nu}\\
&=\sum_{(\a,\b)} \sum_{\{(\mu, \nu), \b=\mu\}} A_{\a\b}A_{\mu\nu}.
\end{split}
\end{equation} On the other hand, we have 
\begin{equation}
\label{mm}
\begin{split}
UV
&=\bigsqcup_{(\a, \b)} U_{\a\b}\bigsqcup_{(\mu, \nu)}V_{\mu\nu}\\
&=\bigsqcup_{(\a, \b)}\bigsqcup_{(\mu, \nu)} U_{\a\b}V_{\mu\nu}\\
&=\bigsqcup_{(\a, \b)}\bigsqcup_{\{(\mu, \nu), \b=\mu\}} U_{\a\b}V_{\mu\nu},
\end{split}
\end{equation} where the last equality follows from the multiplication of the groupoid $\mG_\G$. If all the pairs $(\a, \nu)$ appearing in \eqref{mm} are pairwise distinct, then by comparing \eqref{multi} and \eqref{mm} we obtain $t_{UV}=t_Ut_V$, since $1_{U'_{\a\b}V'_{\mu\nu}}=1_{U'_{\a\b}}*1_{V'_{\mu\nu}}$ in $A_R(\mG)$. Assume that there are pairs $(\a_1, \b_1)\neq (\a_2,\b_2)$ and $(\mu_1,\nu_1)\neq (\mu_2,\nu_2)$ such that $(\a_1,\nu_1)=(\a_2,\nu_2)$. In this case, we have $\b_1=\mu_1$, $\b_2=\mu_2$, $\b_1\neq \b_2$ and $\mu_1\neq \mu_2$. Since the ranges of $U_{\a_1\b_1}$ and $U_{\a_2\b_2}$ are disjoint, so are $U_{\a_1\b_1}V_{\mu_1\nu_1}$ and $U_{\a_2\b_2}V_{\mu_2\nu_2}$. Observe that the matrix $A_{\a_1\b_1}A_{\mu_1\nu_1}+A_{\a_2\b_2}A_{\mu_2\nu_2}$ has $1_{U_{\a_1\b_1}V_{\mu_1\nu_1}}+1_{U_{\a_2j\b_2}V_{\mu_2\nu_2}}$ in the $(\a_1, \nu_1)$-th position and all its other entries are zero. After combining pairs where $(\a_i,\nu_i)=(\a_{i'},\nu_{i'})$ in \eqref{mm}, we have $t_{UV}=t_Ut_V$. For (R3), suppose that $U$ and $V$ are disjoint elements in $B^{\rm co}_\g(\mG_\G)$ with $\g\in\G$ such that $U\cup V$ is a bisection. Write $U=\bigsqcup_{(\a,\b)} U_{\a\b}$ and $V=\bigsqcup_{(\mu,\nu)}V_{\mu\nu}$ as above. It follows that $U\cup V=\bigsqcup_{(\a,\b)}\bigsqcup_{(\mu,\nu)} U_{\a\b}\cup V_{\mu\nu}$. Since $U\cup V$ is a bisection, we have $U'_{\a\b}\cap V'_{\mu\nu}=\emptyset$
if $(\a,\b)=(\mu,\nu)$. In this case, $A_{\a\b}+A_{\mu\nu}$ is the matrix with $1_{U'_{\a\b}\cup V'_{\mu\nu}}=1_{U'_{\a\b}}+1_{V'_{\mu\nu}}$ in the $(\a, \b)$th-position. This shows that after combining pairs that $(\a, \b)=(\mu,\nu)$, we have $t_{U\cup V}=t_U+t_V$. 

By the universality of Steinberg algebra, there exists an $R$-algebra homomorphism $\phi: A_R(\mG_\G)\xra \M_{\G}(A_R(\mG))(\G)$ such that $\phi(1_U)=\sum_{(\a,\b)\in\G\times\G} A_{\a\b}$ given in \eqref{deffor} for each $U\in B^{\rm co}_*(\mG_\G)$. The homomorphism $\phi$ preserves the grading.  

It remains to show that $\phi$ is an isomorphism. For the surjectivity of $\phi$, for any $\a, \b\in\G$, take a matrix $E_{\a\b}\in \M_{\G}(A_R(\mG))(\G)$ with $f'\in A_R(\mG)$ in the $(\a,\b)$-th position and all the other entries zero. It suffices to show that $E_{\a\b}\in {\rm Im} \phi$. We can write $f'=\sum_{w\in F} r_w1_{B_w}$ with $F$ a finite set, $r_w\in R$ and $B_w\in B_*^{\rm co}(\mG)$. Let $X_w=\{g_{\a, \b}\;|\; g\in B_w\}$ for each $w\in F$. Then $X_w$ is a graded compact open bisection of $\mG_\G$. Observe that $\phi(\sum_{w\in F}r_w1_{X_w})=\sum_{w\in F}\phi(r_w1_{X_w})=E_{\a\b}\in {\rm Im} \phi$.

To check the injectivity of $\phi$, suppose that $f''\in A_R(\mG_\G)$ such that $\phi(f'')=0$. Recall from \eqref{express} that we can write $f''=\sum_{i}r_i1_{U_i}$ with $r_i\in R$ and $U_i\in B_*^{\rm co}(\mG_\G)$ mutually disjoint graded compact open bisections. For any $\a,\b\in\G$, let $Y_i=U_i\cap \{g_{\a\b}\;|\; g\in \mG\}$ for each $i$. There exists $Y_i'\subseteq \mG$ such that $Y_i=\{g_{\a\b}\;|\; g\in Y_i'\}$. Observe that $Y_{i}'\cap Y_{i'}'=\emptyset$ for $i\neq i'$; otherwise, $U_i\cap U_{i'}\neq \emptyset$. We write $\phi(f'')=E\in \M_{\G}(A_R(\mG))(\G)$. It follows that the entry of $E$ in the $(\a, \b)$-th position is $\sum_ir_i1_{Y_{i}'}\in A_R(\mG)$. Thus $\sum_ir_i1_{Y_{i}'}=0$ in $A_R(\mG)$. Since all the $Y_i'$'s are disjoint, we obtain $r_i=0$ for all $i$, implying $f''=0$ in $A_R(\mG_\G)$. Hence, $\phi$ is injective.
 \end{proof}

\begin{prop} \label{stegrrecon}
Let $\G$ be a discrete group, $\mG$ a $\G$-graded ample groupoid and $R$ a commutative ring with unit.  Suppose that $f:\mG^{(0)}\xra \Z$ with $\Z$ a discrete group is a constant function such that $f(x)=n$ with $n$ a positive integer for any $x\in \mG^{(0)}$.
Take $(\g_1,\g_2, \cdots, \g_{n+1})\in\G^{n+1}$. Let $\psi: \mG_f^{(0)}\xra \G$ be the continuous map given by $\psi(x_{ii})=\g_{i+1}$ for $x\in \mG^{(0)}$ and $0\leq i\leq n$. Then \[A_R(\mG_f(\psi))\cong \M_{n+1}(A_R(\mG))(\g_1,\g_2, \cdots, \g_{n+1}),\] as $\G$-graded $R$-algebras.
\end{prop}

\begin{proof} 
We observe that $\mG_{f}(\psi)\xra\mG_\G, x_{ij}\mapsto x_{\g_{i+1}^{-1}\g_{j+1}^{-1}}$ is a $\G$-graded isomorphism as groupoids. The proof now follows from  Proposition \ref{algebrafor}. 
\end{proof}

If $\G$ is a trivial group, we have the non-graded version of the above result.

\begin{cor} Let $\mG$ be an ample  groupoid and $R$ a commutative ring with unit.  Suppose that $f:\mG^{(0)}\xra \Z$ with $\Z$ a discrete group is a constant function such that $f(x)=n$ with $n$ a positive integer for any $x\in \mG^{(0)}$. Then $A_R(\mG_f)\cong \M_{n+1}(A_R(\mG))$.
\end{cor}

\subsection{Steinberg algebras and graded Morita context}
Let $\mG$ be a $\G$-graded ample groupoid and $ R$ a commutative ring with unit. Recall that the characteristic functions of compact open subsets of $\mG^{(0)}$ are graded local units for the Steinberg algebra $A_R(\mG)$.

We show that the graded $(\mG, \HH)$-equivalence induces a graded Morita equivalence on the Steinberg algebra level. 
 Let $\mG$ and $\HH$ be $\G$-graded ample groupoids. Suppose that $Z$ is a $\G$-graded $(\mG,\HH)$-equivalence with linking
groupoid $L$. Let $i$ denote the inclusion maps from $A_R(\mG)$ and $A_R(\HH)$  into $A_R(L)$. Define 
$$M :=\big \{f \in A_R(L)\;|\; \supp(f) \sub Z \big \} \, \, \text{ and } \, \,  N :=\big \{ f \in A_R(L)\;|\;  \supp(f) \sub Z^{\rm op}\big \}.$$
Let $A_R(\mG)$ and $A_R(\HH)$ act on the left and right of $M$ and on the right and left of $N$ by $a\cdot f = i(a)*f$ and $f\cdot a = f* i(a)$. Then there are $\G$-graded bimodule homomorphisms
$$\psi:M\otimes_{i(A_R(\HH))}N\xra A_R(\mG) \, \, \text{ and } \, \,  \varphi:N\otimes_{i(A_R(\mG))}M\xra A_R(\HH)$$ determined by     
$i \psi(f\otimes g) =f*g$ and $i \varphi(g\otimes f) =g*f$.

We are in a position to relate the graded equivalence of groupoids to equivalence of their associated Steinberg algebras. 

\begin{thm}\label{hfhgwolfsheim}
 Let $\mG$ and $\HH$ be $\G$-graded ample groupoids. Suppose that $Z$ is a $\G$-graded $(\mG,\HH)$-equivalence with $\G$-graded linking groupoid $L$.
The tuple $(A_R(\mG), A_R(\HH), M, N, \psi, \varphi)$ is a surjective graded Morita context, and so $A_R(\mG)$ and $A_R(\HH)$ are graded Morita equivalent.
\end{thm}

\begin{proof}
We refer to \cite[Theorem 5.1]{cs} for the proof  that $(A_R(\mG), A_R(\HH), M, N, \psi, \varphi)$ is a surjective Morita context. It is evident that the tuple $(A_R(\mG), A_R(\HH), M, N, \psi, \varphi)$ is graded surjective Morita context. By \cite[Theorem 2.6]{haefner}, $A_R(\mG)$ and $A_R(\HH)$ are graded Morita equivalent as $\G$-graded rings.
\end{proof}

\section{Graded homology of \'etale groupoids}\label{bellesebasatian}

In this section, we introduce graded homology groups for a graded \'etale groupoid $\mG$ and prove that when the groupoid is strongly graded and ample, its graded homology groups are isomorphic to the homology groups for the $\varepsilon$-th component $\mG_{\varepsilon}$ (Theorem \ref{strongly}). 
 We then establish a short exact sequence \[H_0^{\gr}(\mG)\longrightarrow H_0^{\gr}(\mG) \longrightarrow H_0(\mG) \longrightarrow 0,\] for a graded ample groupoid $\mG$ (Theorem~\ref{ses}). 
This is in line with the van den Bergh exact sequence for graded $K$-theory of a $\Z$-graded regular Noetherian ring $R$ (\cite[\S6.4]{grbook}) 
\[ \longrightarrow K_n^{\gr}(R)\longrightarrow K_n^{\gr}(R)\longrightarrow K_{n}(R) \longrightarrow K_{n-1}^{\gr}(R) \longrightarrow \cdots,\]
and suggests that we might have a long exact sequence
\[ \longrightarrow H_n^{\gr}(\mG)\longrightarrow H_n^{\gr}(\mG)\longrightarrow H_{n}(\mG) \longrightarrow H_{n-1}^{\gr}(\mG) \longrightarrow \cdots .\]

In Section~\ref{sectionsix} we show that for the graph groupoid $\mG_E$ associated to an arbitrary graph $E$, we have $H_0^{\gr}(\mG_E) \cong K_0^{\gr}(L_R(E))$, where $L_R(E)$ is the Leavitt path algebra with the coefficient field $R$.

\begin{thm}\label{KakutaniMori}
Let $R$ be a commutative ring with unit and let $\mG_1$ and $\mG_2$ be $\G$-graded ample groupoids. If $\mG_1$ and $\mG_2$ are $\G$-graded Kakutani equivalent, then the categories $\Gr_R \mG_1$ and $\Gr_R \mG_2$ are graded equivalent. 
\end{thm}
\begin{proof}

There are full clopen subsets $Y_i\subseteq \mG_i^{(0)}$, for $i=1,2$, such that $\mG_1|_{Y_1}$ is graded isomorphic to $\mG_2|_{Y_2}$. By ~\cite[Proposition~3.3]{clarkhazratrigby} the categories  $\Gr_R \mG_1$ and $\Gr A_R(\mG)$ are equivalent. Since $1_{Y_{i}}$ are homogeneous full idempotent of $A_R(\mG_i)$,  and $1_{Y_{i}} A_R(\mG_i) 1_{Y_{i}}  \cong A_R(\mG_i|_{Y_i})$, where $i=1,2$, we have 
\begin{equation*}
\begin{split}
 \Gr_R \mG_1 & \cong \Gr A_R(\mG_1) \cong  \Gr 1_{Y_{1}} A_R(\mG_1) 1_{Y_{1}}  \cong \Gr A_R(\mG_1|_{Y_1})\\
 & \cong \Gr A_R(\mG_2|_{Y_2})
 \cong   \Gr 1_{Y_{2}} A_R(\mG_2) 1_{Y_{2}} \\
 &\cong    \Gr A_R(\mG_2) \cong  \Gr_R \mG_2. \qedhere
\end{split}
\end{equation*} 
\end{proof}

One can  provide an alternative proof for Theorem~\ref{KakutaniMori} as follows. There are full clopen subsets $Y_i\subseteq \mG_i^{(0)}$, for $i=1,2$, such that $\mG_1|_{Y_1}$ is graded isomorphic to $\mG_2|_{Y_2}$. By \cite[Lemma 6.1]{cs} $s^{-1}(Y_1)$ is $\G$-graded $(\mG_1, \mG_1|_{Y_1})$-equivalence. Similarly $r^{-1}(Y_2)$ is $\G$-graded $(\mG_2, \mG_2|_{Y_2})$-equivalence. Graded groupoid equivalence is an equivalence relation. It follows that $\mG_1$ and $\mG_2$ are graded groupoid equivalent. By Theorem \ref{hfhgwolfsheim} and \cite[Theorem 2.6]{haefner} we have the category equivalece $\Gr A_R(\mG_1)\cong \Gr A_R(\mG_2)$. Recall from~\cite[Proposition~3.3]{clarkhazratrigby} that the categories  $\Gr_R \mG_1$ and $\Gr A_R(\mG)$ are equivalent. Thus $\Gr_R \mG_1\cong \Gr A_R(\mG_1)\cong \Gr A_R(\mG_2)\cong\Gr_R \mG_1$.

Since $\mG$ and $\mG_\G$ are graded Kakutani and the latter groupoid is strongly graded (Lemma~\ref{prinp}), combining Theorem~\ref{KakutaniMori} and 
\cite[Theorem 3.10]{clarkhazratrigby}, we obtain
\begin{equation}\label{keymom}
\Gr_R \mG \cong \Gr_R \mG_\G \cong \Modd_R(\mG_\G)_\varepsilon \cong \Modd_R \mG\times_c \G.
\end{equation}

The equivalence~(\ref{keymom}) justifies the use of $\mG\times_c \G$ in order to define the graded version of homology of the $\G$-graded groupoid $\mG$.

\subsection{Graded homology groups for \'etale groupoids}\label{subsection41}

Let $X$ be a locally compact Hausdorff space and $R$ a topological abelian group. Let $\G$ be a discrete group acting from right continuously on $X$, i.e.,   there is a group homomorphism $\G\xra {\rm Home}(X)^{\rm op}$, where ${\rm Home}(X)$ consists of homeomorphisms from $X$ to $X$ with the multiplication given by the composition of maps and ${\rm Home}(X)^{\rm op}$ is the opposite group of ${\rm Home}(X)$. Denote by $C_c(X,R)$  the set of $R$-valued continuous functions with compact support. With point-wise addition, $C_c(X,R)$ is an abelian group. For $\gamma \in \Gamma$ and $f\in C_c(X,R)$, defining $\gamma f (x):=f(x\gamma)$, makes $C_c(X,R)$ a left $\G$-module. 

Let $\pi:X\to Y$ be a local homeomorphism 
between locally compact Hausdorff spaces which respect the right $\G$-action, namely $\pi(x \gamma)=\pi(x)\gamma$.  
For $f\in C_c(X,R)$, define the map $\pi_*(f):Y\to R$ by 
\[
\pi_*(f)(y)=\sum_{\pi(x)=y}f(x). 
\]
One can check that $\pi_*$ satisfies $\g\pi_*(f)=\pi_*(\g f)$ for $f\in C_c(X, R)$. Thus $\pi_*$ is a $\Gamma$-homomorphism from $C_c(X,R)$ to $C_c(Y,R)$ which makes $C_c(-,R)$  a functor from the category of right $\G$-locally compact Hausdorff spaces with $\G$-local homeomorphisms to the category of left $\G$-modules.  

For a groupoid $\mG$, we denote by ${\rm Aut}(\mG)$ the group of homeomorphisms between 
$\mG$ which respect the composition of $\mG$. We say that a discrete group $\G$ continuously acts on $\mG$ from right side if $\alpha: \G\xra {\rm Aut}(\mG)^{\rm op}$ is a group homomorphism with ${\rm Aut}(\mG)^{\rm op}$ the opposite group of ${\rm Aut}(\mG)$.

Let $\mG$ be an \'etale groupoid and $\G$ a discrete group acting continuously on $\mG$ from right side. For $n\in\N$, we write $\mG^{(n)}$ 
for the space of composable strings of $n$ elements in $\mG$, that is, 
\[
\mG^{(n)}=\{(g_1,g_2,\dots,g_n)\in\mG^n\mid
s(g_i)=r(g_{i+1})\text{ for all }i=1,2,\dots,n{-}1\}. 
\]
For $i=0,1,\dots,n$, with $n\geq 2$
we let $d_i:\mG^{(n)}\to \mG^{(n-1)}$ be a map defined by 
\[
d_i(g_1,g_2,\dots,g_n)=\begin{cases}
(g_2,g_3,\dots,g_n) & i=0 \\
(g_1,\dots,g_ig_{i+1},\dots,g_n) & 1\leq i\leq n{-}1 \\
(g_1,g_2,\dots,g_{n-1}) & i=n. 
\end{cases}
\]
When $n=1$, we let $d_0,d_1:\mG^{(1)}\to\mG^{(0)}$ be 
the source map and the range map, respectively. 
Clearly the maps $d_i$ are local homeomorphisms which respects the right $\G$-actions.

Define the $\G$-homomorphisms $\partial_n:C_c(\mG^{(n)},R)\to C_c(\mG^{(n-1)},R)$ 
by 
\begin{equation} \label{hdtfjdjjd}
\partial_1=s_*-r_* \text{~~~~and~~~~}
\partial_n=\sum_{i=0}^n(-1)^id_{i*}. 
\end{equation}
One can check that the sequence 
\begin{equation}\label{moorcomlex}
0\stackrel{\partial_0}{\longleftarrow}
C_c(\mG^{(0)},R)\stackrel{\partial_1}{\longleftarrow}
C_c(\mG^{(1)},R)\stackrel{\partial_2}{\longleftarrow}
C_c(\mG^{(2)},R)\stackrel{\partial_3}{\longleftarrow}\cdots
\end{equation}
is a chain complex of left $\G$-modules.

The following definition comes from \cite{crainicmoerdijk,matui} with an added structure of a group acting on the given groupoid from right.

\begin{deff}\label{homology}({\it Homology groups of a groupoid $\mG$})
Let $\mG$ be an \'etale groupoid and $\G$ a discrete group acting continuously from right on $\mG$. Define the homology groups of $\mG$ with coefficients $R$, $H_n(\mG,R)$, $n\geq 0$, to be the homology groups of the Moore complex~(\ref{moorcomlex}), 
i.e., $H_n(\mG,R)=\ker\partial_n/\Ima\partial_{n+1}$, which are left $\G$-modules.  
When $R=\Z$, we simply write $H_n(\mG)=H_n(\mG,\Z)$. 
In addition, we define 
\[
H_0(\mG)^+=\{[f]\in H_0(\mG)\mid f(x)\geq0\text{ for all }x\in\mG^{(0)}\}, 
\]
where $[f]$ denotes the equivalence class of $f\in C_c(\mG^{(0)},\Z)$. 
\end{deff}

We are in a position to define  a graded homology theory for a graded groupoid.

\begin{deff}\label{homology}({\it Graded homology groups of a groupoid $\mG$})
Let $\mG$ be a $\G$-graded \'etale groupoid. Then the skew product $\mG \times_c \G$ is an \'etale groupoid with a right $\G$-action 
$(g,\alpha)\gamma=(g,\alpha\gamma)$. We define the graded homology groups of $\mG$ as 
\[
H_n^{\gr}(\mG,R):=H_n(\mG\times_c \G,R), n\geq 0
.\]
\end{deff} When $R=\Z$, we simply write $H^{\gr}_n(\mG)=H_n^{\gr}(\mG,\Z)$. In addition, we define 
\[
{H_0^{\gr}(\mG)}^+={H_0(\mG\times_c\G)}^+.\] Note that $H_n^{\gr}(\mG,R)$ is a left $\G$-module.

For a strongly $\G$-graded ample groupoid $\mG$, we prove that the graded homology $H_n^{\gr}(\mG)$ of $\mG$ is isomorphic to the homology $H_n(\mG_{\varepsilon})$ for the groupoid $\mG_{\varepsilon}$, where $\varepsilon$ is the identity for the group $\G$. 


Recall that a space is $\sigma$-compact if it has a countable cover by compact sets. 

\begin{thm}\label{strongly} Let $\mG$ be a strongly $\G$-graded ample groupoid such that $\mG^{(0)}$ is $\sigma$-compact. Then $H_n^{\gr}(\mG)\cong H_n(\mG_{\varepsilon})$ for each $n\geq 0$, where $\varepsilon$ is the identity of the group $\G$. 
\end{thm}
\begin{proof} By Theorem \ref{tfse}, $\mG_{\varepsilon}$ and $\mG\times_c \G$ is Kakutani equivalent with respect to some full clopen subset of $\mG^{(0)}$ in $\mG_{\varepsilon}$ and $Y_{\varepsilon}=\{(y, \varepsilon)\;|\; y\in Y\}$ a full clopen subset of $\mG^{(0)}\times \G$ in $\mG\times_c \G$. Since $\mG$ is ample, $\mG_{\varepsilon}$ and $\mG\times_c \G$ are ample groupoids. By \cite[Theorem 3.6(2)]{matui} and \cite[Proposition 3.5(2)]{matui}, we have 
$H_n(\mG_{\varepsilon})\cong H_n(\mG_{\varepsilon}|_{Y})$ and 
$H_n(\mG\times_c \G)\cong H_n((\mG\times_c\G)|_{Y_{\varepsilon}})$ for each $n\geq 0$. Since $\mG_{\varepsilon}|_{Y}\cong (\mG\times_c\G)|_{Y_{\varepsilon}}$ as groupoids, it follows that $$H_n(\mG\times_c \G)\cong H_n((\mG\times_c\G)|_{Y_{\varepsilon}})\cong H_n(\mG_{\varepsilon}|_{Y})\cong H_n(\mG_{\varepsilon})$$ for each $n\geq 0$.
\end{proof}



Next we determine the image of the map $\partial_1$ in the Moore complex~\eqref{moorcomlex}.  Here we consider the coefficients of the homology in the ring $R$, with discrete topology. Recall also from~(\ref{hdgftdggd}) that the 
$B_*^{\rm co}(\mG)$ is the collection of all graded compact open bisections of $\mG$.

\begin{lemma} \label{imlemma}Let $\mG$ be a $\G$-graded ample groupoid. We have 
\[{\rm Im} \partial_1=R\-{\rm span}\big \{1_{s(U)}-1_{r(U)}\;|\; U\in B_*^{\rm co}(\mG)\big \}\] as an abelian subgroup of $C_c(\mG^{(0)}, R)$.
\end{lemma}
\begin{proof} Since $\mG$ is an ample groupoid, by Lemma~\ref{copenhagenairport} it suffices to show that $\partial_1(1_U)=1_{s(U)}-1_{r(U)}$, for each $U\in B_*^{\rm co}(\mG)$. For each $x\in \mG^{(0)}$, we have 
\begin{align*}
\partial_1(1_U)(x)
&=\sum_{\{g\in \mG\;|\; s(g)=x\}}1_U(g)-\sum_{\{g\in \mG\;|\; r(g)=x\}}1_U(g)\\
&=\sum_{\{g\in \mG\;|\; s(g)=x\}}1_U(g)-\sum_{\{g\in \mG\;|\; s(g)=x\}}1_U(g^{-1})\\
&=1_{U^{-1}U}(x)-1_{UU^{-1}}(x)\\
&=(1_{s(U)}-1_{r(U)})(x). \qedhere
\end{align*}
\end{proof}

In the next section we will consider the homology groups of graph groupoids which have a natural $\Z$-graded structure. For the remaining of this section, we will work with $\Z$-graded groupoids. For a $\Z$-graded ample groupoid, we establish an exact sequence relating the graded zeroth homology group to the non-graded version (Theorem~\ref{ses}). In order to achieve this, we will use the calculus of smash products on the level of Steinberg algebras first studied in~\cite{ahls} (see also~\S\ref{intheairfaroe}).  
 
 The following fact is needed in the proof of Theorem~\ref{ses} and we record it here. 

\begin{lemma} \label{general}Let $G$ be an abelian group. Then the following sequence of abelian groups is exact
\begin{align*}
\label{sexact}
\CD
    \bigoplus_{\Z}G@>{\Lambda}>>  \bigoplus_{\Z}G @>{\Sigma}>>  G@>>>0, 
\endCD
\end{align*} where $\Lambda(\{x_i\}_{i\in \Z})=\{x_{i-1}\}_{i\in\Z}-\{x_i\}_{i\in\Z}$ and $\Sigma(\{x_i\}_{i\in\Z})=\sum_i x_i$ for any $\{x_i\}_{i\in\Z}$ in $\bigoplus_{\Z}G$.
\end{lemma}

\begin{proof}  Obviously, $\Sigma$ is surjective and $\Sigma\circ \Lambda=0$. It suffices to show that $\Ker \Sigma\sub {\rm Im} \Lambda$. Take  $\{x_i\}_{i\in\Z}\in \Ker \Sigma$. Denote by $j$  the minimal integer $i$ such that $x_{i}$ is nonzero. Denote by $k$ the maximal integer $i$ such that $x_{i}$ is nonzero. We have $\sum_{i=j}^{k}x_{i}=0$, since $\{x_i\}_{i\in\Z}\in \Ker \Sigma$. Let $\{y_i\}_{i\in\Z}$ be given by 
\begin{equation*}
y_i=
\begin{cases}-\sum_{l=j}^i x_l, & \text{if~} j\leq i;\\
0, & \text{otherwise.}
\end{cases}
\end{equation*} Observe that if $i>k$, we have $y_i=-\sum_{l=j}^ix_l=-\sum_{l=j}^kx_l-\sum_{l=k+1}^ix_l=0$. Thus we have $\{y_i\}_{i\in\Z}\in \bigoplus_{\Z}G$. By the definition of $\Lambda$, we have $\Lambda(\{y_i\}_{i\in \Z})=\{x_i\}_{i\in\Z}$, implying $\Ker \Sigma\sub {\rm Im} \Lambda$.
\end{proof}

We are in a position to prove the main theorem of this section. 

\begin{thm} \label{ses} Let $\mG$ be a $\Z$-graded ample groupoid. Then we have an exact sequence, 
\begin{align*}\CD
    H_0^{\gr}(\mG)@>{\widetilde{\Lambda}}>>  H_0^{\gr}(\mG) @>{\widetilde{\Sigma}}>>  H_0(\mG)@>>>0.
\endCD
\end{align*} 
\end{thm}

\begin{proof} 

Specialising the calculus of smash products to the case of $\Z$-graded groupoids, by~(\ref{saiso}) there is an isomorphism of $\Z$-graded algebras  $A_R(\mG\times_c \Z)\cong A_R(\mG)\# \Z$ which on the level of unit spaces gives rise to 
\[C_c\big ((\mG\times_c \Z)^{(0)}, \Z\big)\cong C_c(\mG^{(0)}, \Z)\#\Z.\] Observe that the multiplication in $C_c(\mG^{(0)}, \Z)\#\Z$ is given by $(fp_n) (gp_m)=\delta_{n, m}(fg)p_n$, for $f, g\in C_c(\mG^{(0)}, \Z)$ and $n,m\in\Z$ (see~(\ref{morningfaero})).
It follows that 
\begin{equation*}
\label{sigma}
\sigma: C_c\big ((\mG\times_c \Z)^{(0)}, \Z\big)\xrightarrow[]{\sim} C_c(\mG^{(0)}, \Z)\#\Z\xrightarrow[]{\sim} \bigoplus_{\Z}C_c(\mG^{(0)}, \Z)
\end{equation*} 
is an isomorphism of rings, sending $1_U$ to $\{f_k\}_{k\in\Z}$ such that $f_{n_i}=1_{U_i}$ and $f_k=0$ for all other $k$'s with $U=\bigsqcup_{i=1}^l U_i\times \{n_i\}$ satisfying $n_i\neq n_j$ for $i\neq j$. Thus for any $f\in C_c((\mG\times_c \Z)^{(0)}, \Z)$ there is a unique element $\{f_i\}_{i\in\Z}$ in $\bigoplus_{\Z}C_c(\mG^{(0)}, \Z)$ corresponding to $f$. Note that there are only finitely many $i\in \Z$ such that $f_i$ are nonzero.

Setting $G=C_c(\mG^{(0)}, \Z)$, by Lemma~\ref{general}, we the following exact sequence,

\begin{align}\label{exact}
\CD
    \bigoplus_{\Z}C_c(\mG^{(0)}, \Z)@>{\Lambda}>>  \bigoplus_{\Z}C_c(\mG^{(0)}, \Z) @>{\Sigma}>>  C_c(\mG^{(0)}, \Z)@>>>0.
\endCD
\end{align}

We observe that $\Lambda: \bigoplus_{\Z}C_c(\mG^{(0)}, \Z)\xra \bigoplus_{\Z}C_c(\mG^{(0)}, \Z)$ satisfies \begin{equation}
\label{preserve}
\Lambda\sigma({\rm Im} \partial_1)\sub \sigma({\rm Im} \partial_1),\end{equation} where $\partial_1: C_c((\mG\times_c \Z)^{(1)}, \Z)\xra C_c((\mG\times_c \Z)^{(0)}, \Z)$ is given by \eqref{hdtfjdjjd}
for the groupoid $\mG\times_c \Z$. By Lemma \ref{imlemma}, we only need to show that $\Lambda(\sigma(1_{s(U)}-1_{r(U)}))\in {\rm Im}\partial_1$, for $U\in B_n^{\rm co}(\mG\times_c\Z)$ with $n\in\Z$. By \eqref{normform}, we write 
$U = \bigsqcup^l_{i=1} U_i \times \{n_i\}$ with $n_i\neq n_j$ for $i\neq j$. Then we have 
\begin{align*}
\Lambda\sigma(1_{s(U)}-1_{r(U)})
=&\Lambda\sigma(1_{\bigsqcup_{i=1}^ls(U_i)\times \{n_i\}}-1_{\bigsqcup_{i=1}^lr(U_i)\times \{n+n_i\}})\\
=&\sigma(1_{\bigsqcup_{i=1}^ls(U_i)\times \{n_i+1\}}-1_{\bigsqcup_{i=1}^lr(U_i)\times \{n+n_i+1\}})-\sigma(1_{\bigsqcup_{i=1}^ls(U_i)\times \{n_i\}}-1_{\bigsqcup_{i=1}^lr(U_i)\times \{n+n_i\}})\\
=&\sigma(1_{s(U')}-1_{r(U')})-\sigma(1_{s(U)}-1_{r(U)})\in\sigma {\rm Im} \partial_1,
\end{align*} where $U'=\bigsqcup_{i=1}^l U_i\times\{n_i+1\}$. We have \begin{equation}
\label{isop}
\Sigma\sigma({\rm Im}\partial_1)\sub {\rm Im} \partial'_1,
\end{equation} where $\partial'_1: C_c(\mG^{(1)}, \Z)\xra C_c(\mG^{(0)}, \Z)$ is given by \eqref{hdtfjdjjd} for the groupoid $\mG$. Observe that \eqref{isop} follows from the following equalities for $U = \bigsqcup^l_{i=1} U_i \times \{n_i\}\in B_*^{\rm co}(\mG\times_c\Z)$
\begin{align*}
\Sigma\sigma(1_{s(U)}-1_{r(U)})
&=\Sigma\sigma(1_{\bigsqcup_{i=1}^ls(U_i)\times \{n_i\}}-1_{\bigsqcup_{i=1}^lr(U_i)\times \{n+n_i\}})\\
&=\sum_{i=1}^l(1_{s(U_i)}-1_{r(U_i)})\in{\rm Im}\partial'_1.
\end{align*}

The statement that the following sequence of abelian groups is exact 
\begin{align*}
\label{sss}
\CD
    \bigoplus_{\Z}C_c(\mG, \Z)/\sigma({\rm Im}\partial_1)@>{\Lambda}>>  \bigoplus_{\Z}C_c(\mG, \Z)/\sigma({\rm Im}\partial_1) @>{\Sigma}>>  C_c(\mG, \Z)/{{\rm Im}\partial'_1}@>>>0
\endCD
\end{align*}  follows from (\ref{exact}), \eqref{preserve}, \eqref{isop} and the fact that $\Sigma: \sigma({\rm Im}\partial_1)\xra {\rm Im}\partial'_1$ is surjective.

Recall that $H_0^{\gr}(\mG)=C_c((\mG\times_c\Z)^{(0)}, \Z)/{\rm Im}\partial_1$ and $H_0(\mG)=C_c(\mG^{(0)}, \Z)/{\rm Im}\partial'_1$. Let $\widetilde{\Lambda}: H_0^{\gr}(\mG)\xra H_0^{\gr}(\mG)$ be the composition 
\begin{equation*}\label{faero1}
H_0^{\gr}(\mG)\xrightarrow[]{\sigma} \bigoplus_{\Z}C_c(\mG, \Z)/\sigma({\rm Im}\partial_1)\xrightarrow[]{\Lambda}\bigoplus_{\Z}C_c(\mG, \Z)/\sigma({\rm Im}\partial_1)\xrightarrow[]{\sigma^{-1}}H_0^{\gr}(\mG)
\end{equation*}
and $\widetilde{\Sigma}: H_0^{\gr}(\mG)\xra H_0(\mG)$ be the composition 
\begin{equation*}\label{faero2}
H_0^{\gr}(\mG)\xrightarrow[]{\sigma} \bigoplus_{\Z}C_c(\mG, \Z)/\sigma({\rm Im}\partial_1)\xrightarrow[]{\Sigma}C_c(\mG, \Z)/{\rm Im}\partial'_1.
\end{equation*}

The theorem now  follows from the following commutative diagram whose second row is an exact sequence. 
\begin{equation*}
\begin{split}
\xymatrix{ 
	H^{\gr}_0(\mG)\ar[r]^{\widetilde{\Lambda}} \ar[d]^{\sigma}& H_0^{\gr}(\mG) \ar[r]^{\widetilde{\Sigma}} \ar[d]^{\sigma}& H_0(\mG)\ar[r]\ar[d]^{\rm id}&0 \\
	\bigoplus_{\Z}C_c(\mG, \Z)/\sigma({\rm Im}\partial_1)\ar[r]^{\Lambda} &\bigoplus_{\Z}C_c(\mG, \Z)/\sigma({\rm Im}\partial_1) \ar[r]^{\Sigma} & C_c(\mG, \Z)/{\rm Im}\partial'_1 \ar[r]& 0. }
\end{split}
\end{equation*}
\end{proof}

\section{Graded homology and graded Grothendieck group of path algebras}\label{sectionsix}

In this section, we first recall the concept of Leavitt path algebra of a graph $E$ and a basis for a Leavitt path algebra. We prove that a quotient of the diagonal of the Leavitt path algebra for the covering graph of $E$ is isomorphic to the graded Grothendieck group of the Leavitt path algebra $L_{R}(E)$ over a field $R$. This will be used to show that the for the graph groupoid $\mG_E$ associated to an arbitrary graph $E$, we have $H_0^{\gr}(\mG_E) \cong K_0^{\gr}(L_R(E))$.

We briefly recall the definition of a Leavitt path algebra and establish notation. For a comprehensive study of these algebras refer to \cite{lpabook}.

A directed graph $E$ is a tuple $(E^{0}, E^{1}, r, s)$, where $E^{0}$ and $E^{1}$ are
sets and $r,s$ are maps from $E^1$ to $E^0$. We think of each $e \in E^1$ as an edge pointing from $s(e)$ to $r(e)$. We use the convention that a (finite) path $p$ in $E$ is
a sequence $p=\a_{1}\a_{2}\cdots \a_{n}$ of edges $\a_{i}$ in $E$ such that
$r(\a_{i})=s(\a_{i+1})$ for $1\leq i\leq n-1$. We define $s(p) = s(\a_{1})$, and $r(p) =
r(\a_{n})$. 

A directed graph $E$ is said to be \emph{row-finite} if for each vertex $u\in E^{0}$,
there are at most finitely many edges in $s^{-1}(u)$. A vertex $u$ for which $s^{-1}(u)$
is empty is called a \emph{sink}, whereas $u\in E^{0}$ is called an \emph{infinite
emitter} if $s^{-1}(u)$ is infinite. If $u\in E^{0}$ is neither a sink nor an infinite
emitter, we call it a \emph{regular vertex}.

\begin{deff} \label{defleavitt} Let $E$ be a directed graph and $R$ a unital ring. The \emph{Leavitt path algebra} $L_{R}(E)$ of $E$ is the $R$-algebra generated by the set $\{v \mid v\in E^{0}\}\cup \{e \mid e\in E^{1}\}
\cup\{e^{*} \mid e\in E^{1}\}$ subject to the following relations:

\begin{enumerate}[\upshape(1)]
\item $uv=\delta_{u, v}v$ for every $u, v\in E^{0}$;
\smallskip

\item $s(e)e=er(e)=e$ for all $e\in E^{1}$;
\smallskip
\item $r(e)e^{*}=e^{*}=e^{*}s(e)$ for all $e\in E^{1}$;
\smallskip
\item $e^{*}f=\delta_{e, f}r(e)$ for all $e, f\in E^{1}$; and
\smallskip
\item $v = \sum_{e\in s^{-1}(v)}ee^{*}$ for every regular vertex $v\in
    E^{0}$.
\end{enumerate}
\end{deff} 

Let $\G$ be a group with identity $\varepsilon$, and let $w :E^{1}\xra \G$ be a function.
Extend $w$ to vertices and ghost edges by defining  $w(v) = \varepsilon$ for $v\in E^{0}$
and $w(e^{*}) = w(e)^{-1}$ for $e\in E^{1}$. The relations in Definition~\ref{defleavitt}
are homogeneous with respect to $w$ and makes $L_R(E)$ a $\G$-graded ring. The set of all finite sums of vertices in $E^{0}$ is a set of graded local units for $L_R(E)$ 
(see \cite[Lemma 1.2.12]{lpabook}). Furthermore,
$L_{R}(E)$ is unital if and only if $E^0$ is finite.

If $p=\alpha_{1}\alpha_2\cdots\alpha_{n}$ is a path in $E$ of length $n\geq 1$, we define $p^{*}=\alpha_{n}^*\cdots\alpha_2^*\alpha_{1}^*$. We have $s(p^*)=r(p)$ and $r(p^*)=s(p)$. For convention, we set $v^*=v$ for $v\in E^{0}$. We observe that for paths $p, q$ in $E$ satisfying $r(p)\neq r(q)$, $pq^*=0$ in $L_R(E)$. Recall that the Leavitt path algebra $L_R(E)$ is spanned by the following set
$$\big\{v, p, p^*, \eta\gamma^*\; |\; v\in E^0, p, \gamma, \text{~and~} \eta \text{~are ~nontrivial paths in ~} E \text{~with~} r(\gamma)=r(\eta)\big\}.$$
In general, this set is not $R$-linearly independent. 

In \cite{aajz} authors gave a basis for row-finite Leavitt path algebras over a field which can be extended to arbitrary graphs over commutative rings as well~\cite[Theorfem 2.7]{aragoodearl}. We need this result in the paper, thus we recall it here. 
 For each regular vertex $v\in E^0$, we fix an edge (called {\it special edge}) starting at $v$.

\begin{lemma}\cite[Theorem 1]{aajz}
\label{lbasis} 
The following elements form a basis for the Leavitt path algebra $L_R(E)$: 
\begin{enumerate}[\upshape(1)]

\item $v$, where $v\in E^0$;

\smallskip

\item  $p, p^*$, where $p$ is a nontrivial path in $E$;

\smallskip

\item  $qp^*$ with $r(p)=r(q)$, where
$p=\alpha_{1}\cdots\alpha_{m}$ and $q=\beta_{1}\cdots\beta_{n}$ are nontrivial paths of $E$ such that $\alpha_{m}\neq \beta_{n}$, or $\alpha_{m}=\beta_{n}$ which is not special.\hfill $\square$\end{enumerate}
\end{lemma}

For a $\G$-graded ring $A$ one can consider the abelian monoid of isomorphism classes of graded finitely generated projective modules denoted by $\VV^{\gr}(A)$. For the precise definition of $\VV^{\gr}(A)$ for a graded ring with graded local units, refer to \cite[\S 4A]{ahls}. Recall that the shift functor $\mathcal{T}_\a: \Gr A \xra  \Gr A $ restricts to the category of graded finitely generated projective modules. Thus the group $\G$ acts on $\VV^{\gr}(A)$ which makes $\VV^{\gr}(A)$ a $\G$-module. The graded Grothendieck group, $K^{\gr}_0(A)$, is defined as the group completion of $\VV^{\gr}(A)$ which naturally inherits the $\G$-module structure of $\VV^{\gr}(A)$ and thus becomes a $\Z[\G]$-module. Here, $\Z[\G] $ is a group ring.

The monoid $\VV^{\gr}(L_R(E))$ for a $\G$-graded Leavitt path algebra of a graph $E$ over a field $R$ was studied in \cite[\S 5C]{ahls}.
In order to compute the monoid $\VV^{\gr}(L_{R}(E))$ for an arbitrary graph $E$, an abelian monoid $M_{E}^{\gr}$ for an arbitrary graph $E$ was defined in \cite[\S5.3]{ahls} as the free abelian monoid such that generators $\{a_{v}({\g}) \;|\;  v\in
E^{0}, \g\in \G\}$ are supplemented by generators $b_{Z}(\g)$, where $\g\in\G$ and $Z$ runs
through all nonempty finite subsets of $s^{-1}(u)$ for infinite emitters $u\in E^{0}$, subject to the relations

\begin{enumerate}
\item[(1)] $a_{v}({\g})=\sum\limits_{e\in s^{-1}(v)}a_{r(e)}(w(e)^{-1}\g)$  for all
    regular vertices $v\in E^{0}$ and $\g\in\G$;
    
    \smallskip
\item[(2)] $a_{u}({\g})=\sum\limits_{e\in Z}a_{r(e)}(w(e)^{-1}{\g})+b_{Z}({\g})$ for
    all $\g\in\G$, infinite emitters $u\in E^{0}$ and nonempty finite subsets
    $Z\subseteq s^{-1}(u)$;
    
        \smallskip
\item[(3)] $b_{Z_{1}}({\g})=\sum\limits_{e\in Z_{2}\setminus
    Z_{1}}a_{r(e)}({w(e)^{-1}\g})+b_{Z_{2}}({\g})$ for all $\g\in\G$, infinite
    emitters $u\in E^{0}$ and nonempty finite subsets $Z_{1}\subseteq Z_{2}\subseteq
    s^{-1}(u)$.
\end{enumerate} Recall that the group $\Gamma$ acts on the monoid $M_{E}^{\gr}$ as follows. For any $\b\in \G$,
\begin{equation} \label{groupaction1}
    \b\cdot a_{v}({\g})=a_{v}({\b\g})\qquad\text{ and }\qquad \b\cdot b_{Z}({\g})=b_{Z}({\b\g}).
\end{equation} There is a $\G$-module isomorphism 
$\VV^{\gr}(L_R(E))\cong M_E^{\gr}$ (see \cite[Proposition 5.7]{ahls}).

Let $\G$ be a group and $w:E^{1}\xrightarrow{} \G$ a function. As in \cite[\S2]{gr}, the
\emph{covering graph} $\overline{E}$ of $E$ with respect to $w$ is given by
\begin{gather}\label{lifeoflux}
    \overline{E}^{0} = \{v_{\a} \mid  v\in E^{0}\text{ and } \a\in\G\},\qquad
    \overline{E}^{1} = \{e_\a \mid e \in E^1\text{ and } \a \in \G\},\\
    s(e_\a) = s(e)_\a,\quad\text{ and } r(e_\a) = r(e)_{w(e)^{-1}\a}. \notag
\end{gather}

\begin{example}\label{onetwo3}
Let $E$ be a graph and define $w:E^1 \xrightarrow{} \Z$ by $w(e) = 1$ for all $e \in E^1$.
Then $\overline{E}$ (sometimes denoted $E\times_1 \mathbb{Z}$) is given by
\begin{gather*}
    \overline E^0 = \big\{v_n \mid v \in E^0 \text{ and } n \in \Z \big\},\qquad
    \overline E^1 = \big\{e_n \mid e\in E^1 \text{ and } n\in \Z \big\},\\
    s(e_n) = s(e)_n,\qquad\text{ and } r(e_n) = r(e)_{n-1}.
\end{gather*}

As examples, consider the following graphs
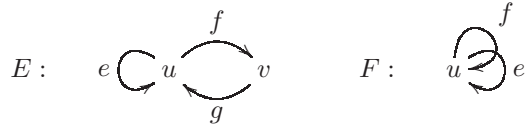
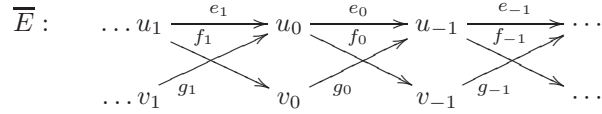
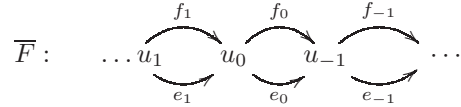
\begin{equation*}
{\def\labelstyle{\displaystyle}
E : \quad \,\, \xymatrix{
 u \ar@(lu,ld)_e\ar@/^0.9pc/[r]^f & v \ar@/^0.9pc/[l]^g
 }} \qquad \quad
{\def\labelstyle{\displaystyle}
F: \quad \,\, \xymatrix{
   u \ar@(ur,rd)^e  \ar@(u,r)^f
}}
\end{equation*}
Then
\begin{equation*}
\overline{E}: \quad \,\,\xymatrix@=15pt{
\dots  {u_{1}} \ar[rr]^-{e_{1}} \ar[drr]^(0.4){f_{1}} &&  {u_{0}} \ar[rr]^-{e_0} \ar[drr]^(0.4){f_0} && {u_{-1}}  \ar[rr]^-{e_{-1}} \ar[drr]^(0.4){f_{-1}} && \cdots\\
\dots {v_{1}}   \ar[urr]_(0.3){g_{1}} && {v_{0}} \ar[urr]_(0.3){g_0}  && {v_{-1}}  \ar[urr]_(0.3){g_{-1}}&& \cdots
}
\end{equation*}
and
\begin{equation*}
\overline{F}: \quad \,\,\xymatrix@=15pt{
\dots  {u_{1}} \ar@/^0.9pc/[r]^{f_{1}} \ar@/_0.9pc/[r]_{e_{1}}  &  {u_{0}} \ar@/^0.9pc/[r]^{f_0} \ar@/_0.9pc/[r]_{e_0} & {u_{-1}}  \ar@/^0.9pc/[r]^{f_{-1}}  \ar@/_0.9pc/[r]_{e_{-1}} & \quad \cdots
}
\end{equation*}
\end{example}

For the rest of the section we work with Leavitt path algebras with integral coefficients. We denote by $E^*$ the set of all finite paths in $E$. Recall that  the diagonal of the Leavitt path algebra $L_\Z(E)$, denoted by $\D_\Z(E)$, is $$\Z\-{\rm span}\{\alpha\alpha^*\;|\; \alpha\in E^*\},$$ which is a commutative subalgebra of $L_\Z(E)$. In particular, $\D_\Z(E)$ is an abelian group. 

\begin{lemma}\label{onthemove} For a graph $E$, the diagonal $\D_\Z(E)$ is the free abelian group generated by symbols $\{\alpha\alpha^*\;|\; \alpha\in E^*\}$ subject to the relation 
\begin{align}
\label{relation}
\alpha\alpha^*=\sum_{e\in s^{-1}(r(\alpha))}\alpha ee^*\alpha^*, 
\end{align} for $\alpha\in E^*$ with $r(\alpha)$ a regular vertex of $E$. 
\end{lemma}
\begin{proof} Denote by $G$ the free abelian group generated by symbols $\{\alpha\alpha^*\;|\; \alpha\in E^*\}$ subject to the relation given in \eqref{relation}. We prove that $G$ and $\D_\Z(E)$ are isomorphic as groups. We define a group homomorphism $f: G\xra \D_\Z(E)$ such that $f(\alpha\a^*)=\a\a^*$ for the generators $\a\a^*$ in $G$. Observe that $f$ is well defined since it preserves the relation in $G$. Obviously, $f$ is surjective. In order to show that $f$ is injective, it suffices to define a group homomorphism $g: \D_\Z(E)\xra G$ such that $gf=\id_{G}$. Recall from Lemma \ref{lbasis} that there is a basis of normal forms for $L_{\Z}(E)$. Since $\D_\Z(E)$ is a subset of $L_{\Z}(E)$, any element in $\D_\Z(E)$ has the normal form $\sum_{i=1}^nm_i\a_i\a_i^*$ such that $\a_i=\a_{i_1}\cdots\a_{i_{l_i}}$ a path in $E$ of length $l_i$ with $\a_{l_i}$ not special and $m_i$ is a nonzero integer. We define $g: \D_\Z(E)\xra G$ by $g(\sum_{i=1}^nm_i\a_i\a_i^*)=\sum_{i=1}^nm_i\a_i\a_i^*$ with $\sum_{i=1}^nm_i\a_i\a_i^*$ a normal form. Note that $g(x+y)=g(x)+g(y)$ for two normal forms in $\D_\Z(E)$. To check that $gf=\id_G$, we only need to show that $gf(\a\a^*)=\a\a^*$ for any generator $\a\a^*$ in $G$. We have the following two cases. If $\a=\a_{1}\cdots\a_{l}$ is a path in $E$ of length $l\geq 0$ with $\a_{l}$ not special, then $gf(\a\a^*)=g(\a\a^*)=\a\a^*$. In the case $l=0$, $\a$ is a vertex in $E$. If $\a=\a_{1}\cdots\a_{l}$ is a path in $E$ of length $l\geq 1$ with $\a_{l}$ special, then we can write $\a\a^*$ as the following normal form $$\a\a^*=\a_1\cdots\a_{k}\a_{k}^*\cdots\a_1^*-\sum_{j=k}^{l-1}\sum_{\{e\in s^{-1}(r(\a_j))\;|\; e\neq \a_{j+1}\}}\a_1\cdots\a_{j}ee^*\a_{j}^*\cdots\a_1^*,$$ where $1\leq k\leq l$ is the number such that $\a_{j}$ is special for any $j>k$. Then we have 
\begin{equation*}
\begin{split}
gf(\a\a^*)
&=g(\a\a^*)\\
&=\a_1\cdots\a_{k}\a_{k}^*\cdots\a_1^*-\sum_{j=k}^{l-1}\sum_{\{e\in s^{-1}(r(\a_j))\;|\; e\neq \a_{j+1}\}}\a_1\cdots\a_{j}ee^*\a_{j}^*\cdots\a_1^*\\
&=\a_1\cdots\a_{l-1}\a_l\a_l^*\a_{l-1}^*\cdots\a_1^*\\
&=\a\a^*, 
\end{split}
\end{equation*} where the second last equality holds because we have $\a_1\cdots\a_j\a_j^*\cdots\a_1^*=\sum_{e\in s^{-1}(r(\a_j))}\a_1\cdots\a_jee^*\a_j^*\cdots\a_1^*$ for each $j=k, k+1, \cdots, l-1$. This finishes the proof. 
\end{proof}

Let $E$ be a graph. We denote by $E^{\infty}$ the set of
infinite paths in $E$. Set
\[
X := E^{\infty}\cup  \{\mu\in E^{*}  \mid   r(\mu) \text{ is not a regular vertex}\}.
\]
Let
\[
\mG_{E} := \{(\a x,|\a|-|\b|, \b x) \mid   \a, \b\in E^{*}, x\in X, r(\a)=r(\b)=s(x)\}.
\]
We view each $(x, k, y) \in \mG_{E}$ as a morphism with range $x$ and source $y$. The
formulas $(x,k,y)(y,l,z)= (x,k + l,z)$ and $(x,k,y)^{-1}= (y,-k,x)$ define composition
and inverse maps on $\mG_{E}$ making it a groupoid with $\mG_{E}^{(0)}=\{(x, 0, x) \mid
x\in X\}$ which we identify with the set $X$.

Next, we describe a topology on $\mG_{E}$. For $\mu\in E^{*}$ define
\[
Z(\mu)= \{\mu x \mid x \in X, r(\mu)=s(x)\}\subseteq X.
\]
For $\mu\in E^{*}$ and a finite $F\subseteq s^{-1}(r(\mu))$, define
\[
Z(\mu\setminus F) = Z(\mu) \setminus \bigcup_{\a\in F} Z(\mu \a).
\]
The sets $Z(\mu\setminus F)$ constitute a basis of compact open sets for a locally
compact Hausdorff topology on $X=\mG_{E}^{(0)}$ (see \cite[Theorem 2.1]{we}).

For $\mu,\nu\in E^{*}$ with $r(\mu)=r(\nu)$, and for a finite $F\subseteq E^{*}$ such
that $r(\mu)=s(\a)$ for $\a\in F$, we define
\[
Z(\mu, \nu)=\{(\mu x, |\mu|-|\nu|, \nu x) \mid  x\in X, r(\mu)=s(x)\},
\]
and then
\[
Z((\mu, \nu)\setminus F) = Z(\mu, \nu) \setminus \bigcup_{\a\in F}Z(\mu\a, \nu\a).
\]
The sets $Z((\mu, \nu)\setminus F)$ constitute a basis of compact open bisections for a
topology under which $\mG_{E}$ is an ample groupoid. By \cite[Example~3.2]{cs},
the map
\begin{equation}
\label{assteinberg}
\pi_{E} : L_{R}(E) \longrightarrow A_{R}(\mG_{E})
\end{equation} with $R$ a commutative ring with unit defined by $\pi_{E}(\mu\nu^{*}-\sum_{\a\in F}\mu\a\a^{*}\nu^{*}) =
1_{Z((\mu,\nu)\setminus F)}$ extends to an algebra isomorphism. We observe that the
isomorphism of algebras in \eqref{assteinberg} satisfies
\begin{equation}
\label{valuation}
\pi_{E}(v)=1_{Z(v)}, \quad\pi_{E}(e)=1_{Z(e, r(e))}, \quad\pi_{E}(e^{*})=1_{Z(r(e), e)},
\end{equation}
for each $v\in E^{0}$ and $e\in E^{1}$. Observe that the isomorphism $\pi_E$ in \eqref{assteinberg} restricts to an isomorphism $D_{\Z}(E)\cong C_c(\mG_E^{(0)}, \Z)$ when $R=\Z$.

If $E$ is any graph, and $w : E^1 \to \G$ any function, we extend $w$ to $E^{*}$ by
defining $w(v) = 0$ for $v \in E^0$, and $w(\a_{1}\cdots \a_{n}) = w(\a_{1})\cdots
w(\a_{n})$. We obtain from \cite[Lemma 2.3]{kp} a continuous cocycle
$\widetilde{w}:\mG_{E}\xra \G$ satisfying
\[
    \widetilde{w}(\a x, |\a|-|\b|, \b x) = w(\a) w(\b)^{-1}.
\] Thus $\mG_E$ is $\G$-graded with the grading map $\widetilde{w}$ and thus we have the skew-product groupoid $\mG_E\times_{\widetilde{w}}\G$. There is a groupoid isomorphism 
\begin{equation}
\label{isoforgroupoid}
\rho:\mG_{\overline{E}}\longrightarrow\mG_E\times_{\widetilde{w}} \G\text{~~(see \cite[\S 5B]{ahls})}.\end{equation}

From now on, let $\G$ be an abelian group. We denote by $\D_\Z(\overline{E})$ the diagonal of the Leavitt path algebra $L_{\Z}(\overline{E})$ of the covering graph $\overline{E}$ with respect to the map $w:E^{1} \rightarrow \G$ (see~(\ref{lifeoflux})).

In order to give $D_{\Z}(\overline{E})$ a $\G$-module structure, we define a right continuous $\G$-action on the skew-product groupoid $\mG_{E}\times_{\widetilde{w}} \G$ such that for $g\in \mG_E, \a,\b\in\G$ 
\begin{equation}
\label{actionforsp}
(g, \a)\cdot \b=(g, \a\b^{-1}).
\end{equation} Note that the above action is a continuous action, since $\G$ is an abelian group. Recall from subsection \ref{subsection41} that there is an induced $\G$-action on $C_c((\mG_E\times_{\widetilde{w}}\G)^{(0)}, \Z)$. Since we have $D_{\Z}(\overline{E})\cong C_c(\mG_{\overline{E}}^{(0)}, \Z)\cong C_c((\mG_E\times_{\widetilde{w}}\G)^{(0)}, \Z)$, there is a $\G$-action on $D_\Z(\overline{E})$ such that 
\begin{equation}
\label{actionond}
(\a_\g\a_{\g}^*)\cdot \b=\a_{\g\b}\a_{\g\b}^*
\end{equation} for $\a\in E^*, \b,\g\in \G$. 

\begin{thm} \label{thmiso} Let $E$ be an arbitrary graph, $\G$ an abelian group, $R$ a field and $w: E^1\xra \G$ a function. Then there is a $\G$-module isomorphism 
$$\D_\Z(\overline{E})/{{\langle {r(\a)_{w(\alpha)^{-1}\g}-\a_{\g}\a^*_{\g}}\rangle}_{\a\in E^*, \g\in\G}}\cong K_0^{\gr}(L_R(E)).$$
\end{thm}

\begin{proof}
We define a map $$f:\D_\Z(\overline{E})\longrightarrow {M_E^{\gr}}^{+}$$ such that $f(\alpha_{\g}\alpha_{\g}^*)=a_{r(\alpha)}(w(\alpha)^{-1}\g)$ for $\alpha\in E^*$. In the case that $\a=v$ is a vetex, $f(v_{\g})=a_v(\g)$. By Lemma~\ref{onthemove}, the map $f$ is well defined since $$f(\alpha_{\g}\alpha_{\g}^*)=a_{r(\alpha)}(w(\alpha)^{-1}\g)=\sum_{e\in s_E^{-1}(r(\alpha))}a_{r(e)}(w(e)^{-1}w(\alpha)^{-1}\g)=f\big (\sum_{e\in s_E^{-1}(r(\alpha))}\alpha_{\g} e_{w(\alpha)^{-1}\g}e_{w(\alpha)^{-1} \g)}^*\alpha_{\g}^*\big )$$ which follows from the relation (1) in $M_E^{\gr}$ for each $\alpha\in E^*$ with $r(\a)$ a regular vertex. We observe that $f(r(\a)_{w(\alpha)^{-1}\g}-\a_{\g}\a^*_{\g})=0$ for $\a\in E^*$ and $\g\in\G$. Then there is an induced group homomorphism 
\begin{equation}\label{ffunction}
f: \D_\Z(\overline{E})/{{\langle {r(\a)_{w(\alpha)^{-1}\g}-\a_{\g}\a^*_{\g}}\rangle}_{\a\in E^*, \g\in\G}}\longrightarrow {M_E^{\gr}}^{+}.\end{equation}

Next we define a monoid homomorphism $g: {M_E^{\gr}}\xra \D_\Z(\overline{E})/{{\langle {r(\a)_{w(\alpha)^{-1}\g}-\a_{\g}\a^*_{\g}}\rangle}_{\a\in E^*, \g\in\G}}$ by  $$g(a_{v}(\g))=v_{\g} \text{~~~~and~~~~} g(b_Z(\g'))=u_{\g'}-\sum_{e\in Z}e_{\g'}e_{\g'}^*,$$ for $v\in E^0$, $Z\subseteq s_{E}^{-1}(u)$ a nonzero finite subset for an infinite emitter $u\in E^0$ and $\g,\g'\in \G$. In order to check that $g$ is well defined, we need to show that the relations (1), (2), (3) of $M_E^{\gr}$ are satisfied. For (1), $$g(a_{v}(\g))=v_{\g}=\sum_{e\in s^{-1}(v)}e_{\g}e_{\g}^*=\sum_{e\in s^{-1}(v)}r(e)_{w(e)^{-1}\g}=g(\sum\limits_{e\in s^{-1}(v)}a_{r(e)}(w(e)^{-1}\g)),$$ for regular vertex $v\in E^0$ and $\g\in\G$. For relation (2) of $M_E^{\gr}$, we have $$g\big (\sum\limits_{e\in Z}a_{r(e)}(w(e)^{-1}\g)+b_{Z}({\g})\big)=\sum\limits_{e\in Z}r(e)_{w(e)^{-1}{\g}}+(u_{\g}-\sum_{e\in Z}e_{\g}e_{\g}^*)=u_{\g}=g(a_u(\g)),$$ for $u\in E^0$ an infinite emitter, $Z\subseteq s^{-1}_E(u)$ a nonzero finite subset and $\g\in \G$. For (3), we have 
\begin{align*}
g\big (\sum\limits_{e\in Z_{2}\setminus
    Z_{1}}a_{r(e)}({w(e)^{-1}\g})+b_{Z_{2}}({\g})\big)
&=\sum\limits_{e\in Z_2\setminus Z_1} r(e)(w(e)^{-1}\g)
+u_{\g}-\sum\limits_{e\in Z_2}e_{\g} e_{\g}^*\\
&=u_{\g}+\big (\sum\limits_{e\in Z_2\setminus Z_1}e_{\g} e_{\g}^* -\sum\limits_{e\in Z_2}e_{\g} e_{\g}^*\big)\\
&=u_{\g}-\sum_{e\in Z_1}e_{\g} e_{\g}^*\\
&=g(b_{Z_1}(\g)),
\end{align*} for all $\g\in\G$, infinite
    emitters $u\in E^{0}$ and nonempty finite subsets $Z_{1}\subseteq Z_{2}\subseteq
    s^{-1}(u)$. Then we have an induced group homomorphism 
 \begin{equation}
    \label{g}
    g: {M_E^{\gr}}^+\xra \D_\Z(\overline{E})/{{\langle {r(\a)_{w(\a)^{-1}\g}-\a_{\g}\a^*_{\g}}\rangle}_{\a\in E^*, \g\in\G}}.
    \end{equation}
    
Observe that $f=g^{-1}$ as group homomorphisms. Indeed, $gf(\alpha_\g\alpha_\g^*)=g(a_{r(\alpha)}(w(\alpha)^{-1}\g))=r(\alpha)_{w(\alpha)^{-1}\g}=\alpha_\g\alpha_\g^*$ in $\D_\Z(\overline{E})/{{\langle {r(\a)_{w(\a)^{-1}\g}-\a_{\g}\a^*_{\g}}\rangle}_{\a\in E^*, \g\in\G}}$ for each $\alpha\in E^*$ and $\g\in \G.$
On the other hand, $fg(a_v(\g))=f(v_{\g})=a_v(\g)$ and $fg(b_Z(\g))=f(u_{\g}-\sum_{e\in Z}e_{\g}e_{\g}^*)=a_u(\g)-\sum_{e\in Z}a_{r(e)}(w(e)^{-1}\g)=b_Z(\g)$. Comparing the actions given in \eqref{groupaction1} and \eqref{actionond}, $f$ is a $\G$-module isomorphism. 
\end{proof}

We are in a position to write the main theorem of this section. 

\begin{thm}\label{sojoy}
Let $E$ be an arbitrary graph, $R$ a field and $w:E^1\xra \Z$ a function such that $w(e)=1$ for each edge $e\in E^1$. Then there is a $\Z[x,x^{-1}]$-module isomorphism 
\begin{equation*}
\label{isomain}
\big (H^{\gr}_0(\mG_E), H_0^{\gr}(\mG_E)^+ \big )\cong \big (K^{\gr}_{0}(L_R(E)), K^{\gr}_{0}(L_R(E))^+\big).
\end{equation*}
Furthermore, If $E$ is finite then, $[1_{\mG_E^0}] \mapsto [L(E)]$. 

\end{thm}
\begin{proof} Observe that by Lemma \ref{imlemma} for the groupoid $\mG_E$ we have
\begin{equation}
\label{image}
\begin{split}
{\rm Im} \partial_1
&=\Z\-{\rm Span}\{1_{s(U)}-1_{r(U)}\;|\; U\in B^{\rm co}_*(\mG_E)\}\\
&=\Z\-{\rm Span}\{1_{s(U)}-1_{r(U)}\;|\; U=\sqcup_{i=1}^n Z(\a_i, \b_i)\in B^{\rm co}_*(\mG_E) \text{~for ~}\a_i,\b_i\in E^*\text{~with~} r(\a_i)=r(\b_i) \text{~for each~} i\}\\
&=\Z\-{\rm Span}\{1_{Z(\b)}-1_{Z(\a)}\;|\; \a, \b\in E^* \text{~with~} r(\a)=r(\b)\}\\
&=\Z\-{\rm Span}\{1_{Z(r(\a))}-1_{Z(\a)}\;|\;\a\in E^*\}.
\end{split}
\end{equation} Here the third equality holds since the intersection of $Z(\a, \b)$ and $Z(\a',\b')$ for $\a, \b,\a',\b'\in E^*$ with $r(\a)=r(\b)$ and $r(\a')=r(\b')$ is either empty or equal to one of them, and the source (or range) of $\sqcup_{i} Z(\a_i, \b_i)$ is $\sqcup_iZ(\b_i)$ (or $\sqcup_iZ(\b_i)$). Recall that $\pi_E$ in \eqref{assteinberg} restricts to an isomorphism $D_{\Z}(E)\cong C_c(\mG_E^{(0)}, \Z)$ and sends $r(\a)-\a\a^*$ to $1_{Z(r(\a))}-1_{Z(\a,\a)}=1_{Z(r(\a))}-1_{Z(\a)}$. Then by \eqref{image}, we have $H_0(\mG_E)=C_c(\mG^{(0)}, \Z)/{\rm Im}\partial_1\cong D_\Z(E)/{\langle r(\a)-\a\a^* \rangle}_{\a\in E^*}$. Similarly for $\overline{E}$, we have 
\begin{equation}
\label{diagnaliso}
H_0(\mG_{\overline{E}})\cong D_\Z(\overline{E})/\langle r(\a)_{w(\a)^{-1}\g}-\a_{\g}\a_{\g}^* \rangle_{\a\in E^*, \g\in\G}.
\end{equation} Observe that the $\Z$-action on $C_c(\mG_E^{(0)}\times \Z, \Z)$ is given by $(\b\cdot f)(x)=f(x\cdot \b)$ for $x\in \mG_E^{(0)}\times \Z, f\in C_c(\mG_E^{(0)}\times \Z, \Z),$ and $\b\in\Z$, where $x\cdot \b$ is the action given in \eqref{actionforsp}. To show that the isomorphism given in \eqref{diagnaliso} perserves the $\Z$-module structure, it suffices to show that $\b\cdot 1_{Z(\a)\times\{\g\}}=1_{Z(\a)\times\{\g+\b\}}$ for $\a\in E^*, \b,\g\in\Z$, which can be checked directly.  Combining Theorem \ref{thmiso}, 
\begin{equation}
\label{mainthm}
H_0^{\gr}(\mG_E)=H_0(\mG_E\times_{\widetilde{w}}\Z)\cong D_\Z(\overline{E})/{\langle r(\a)_{w(\a)^{-1}\g}-\a_{\g}\a_{\g}^* \rangle}_{\a\in E^*, \g\in\G}\cong K_0^{\gr}(L_R(E)),\end{equation} 
are $\Z$-module isomorphisms and thus are $\Z[x, x^{-1}]$-module isomorphisms.

Recall that $\VV^{\gr}(L_R(E))\cong M_E^{\gr}$ is the positive cone of $K_0^{\gr}(L_R(E))$. We first show that the isomorphism 
$$\Phi: K_0^{\gr}(L_R(E))\xra H_0^{\gr}(\mG_E),$$ given in \eqref{mainthm} carries $K_0^{\gr}(L_R(E))^+$ to $H_0^{\gr}(\mG_E)^+.$ By \eqref{g}, \eqref{valuation} and \eqref{isoforgroupoid}, $$\Phi(a_v(\g))=1_{Z(v)\times \{\g\}}\in H_0^{\gr}(\mG_E)^+$$ and \begin{equation*}
\begin{split}
\Phi(b_Z(\g'))
&=\rho(\pi_{\overline{E}}(u_{\g'}-\sum_{e\in Z}e_{\g'}e_{\g'}^*))\\
&=1_{Z(u)\times\{\g'\}}-\sum_{e\in Z}1_{Z(e)\times\{\g'\}}\\
&=1_{Z(u)\setminus (\sqcup_{e\in Z}Z(e)) \times\{\g'\}}\\
&=1_{\sqcup_{e\in s^{-1}(u)\setminus Z}Z(e)) \times\{\g'\}}\in H_0^{\gr}(\mG_E)^+.
\end{split}
\end{equation*} Conversely we show that the isomorphism $\Phi^{-1}: H_0^{\gr}(\mG_E)\xra K_0^{\gr}(L_R(E))$ carries $H_0^{\gr}(\mG_E)^+$ to $K_0^{\gr}(L_R(E))^+$. If $[h]\in H_0^{\gr}(\mG_E)^+$, then $h\in C_c(\mG_E^{(0)}\times\Z, \Z)$ with $h(x)\geq 0$ for any $x\in \mG_E^{(0)}\times\Z$. We may write $h=\sum_{i}n_i1_{Z(\a_i)\times \{\g_i\}}$ with $n_i\geq 0$ and
 $Z(\a_i)\times \{\g_i\}$'s mutually disjoint. By \eqref{ffunction} and \eqref{diagnaliso}, we have 
 $$\Phi^{-1}(h)=\sum_in_i f({\a_i}_{\g_i}{{\a_i}_{\g_i}}^*)=\sum_in_i a_{r(\a_i)}(w(\a_i)^{-1}\g_i)\in K_0^{\gr}(L_R(E))^+.$$  
This completes the proof. 
\end{proof}


Theorem~\ref{sojoy} allows us to use the graded homology as a capable invariant in symbolic dynamics. Recall that two-sided shift spaces $X$ and $Y$ are called eventually conjugate if $X^n$ is conjugate to $Y^n$, for all sufficiently large $n$ (\cite[Definition~7.5.14]{lindmarcus}). For a two-sided shift $X$, we denote by $X^{+}$ the one-sided shift associated to $X$, namely, $X^{+}=\{x_{[0,\infty]} \mid x \in X\}$ (see~\cite{carlseneilers} for a summary of concepts on symbolic dynamics).

\begin{thm}\label{colmorgen}
Let $X$ and $Y$ be two-sided shift of finite types. Then $X$ is eventually conjugate to $Y$ if and only if there is an order preserving  $\Z[x,x^{-1}]$-module isomorphism $H_0^{\gr}(\mG_{X^{+}}) \cong H_0^{\gr}(\mG_{Y^{+}})$, where $\mG_{X^{+}}$ and $\mG_{Y^{+}}$ are groupoids associated to one-sided shift $X^{+}$ and $Y^{+},$ respectively.  
\end{thm}
\begin{proof}
Let $E$ and $F$ be the graphs such that $X$ and $Y$ are conjugate to $X_E$ and $X_F$, respectively. Then $X$ is eventually conjugate to $Y$ if and only if there is an isomorphism between Krieger's dimension groups, namely, $(\Delta _{A_E},\Delta_{A_E}^{+},\delta_{A_E})\cong (\Delta _{A_F},\Delta_{A_F}^{+},\delta_{A_F})$, where $A_E$ and $A_F$ are the adjacency matrices of $E$ and $F$ (see \cite[Theorem~7.5.8]{lindmarcus}). On the other hand, by \cite[Theorem~3.11.7]{grbook}, $(K_0^{\gr}(L_R(E), K_0^{\gr}(L_R(E)^+)\cong (\Delta _{A_E},\Delta_{A_E}^{+})$ which respects the $\Z[x,x^{-1}]$-module structure. Note that the groupoid $\mG_{X^+}$ associated to the one-sided shift $X^+$  is isomorphic to the graph groupoid $\mG_E$. Combining this with Theorem~\ref{sojoy} the result follows. 
\end{proof}

\section{Acknowledgements}
This research was supported by the Australian Research Council grant DP160101481. The authors benefitted a great deal from the informative workshops in topological groupoids in Copenhagen in October 2017 and Faroe Islands in May 2018. They would like to thank the organisers  Kevin Brix, Toke Carlsen, S\o ren Eilers  and Gunnar Restorff.


\begin{thebibliography}{9999}



  
\bibitem{lpabook} G. Abrams, P. Ara, M. Siles Molina, Leavitt path algebras, Lecture Notes in Mathematics {\bf 2191}, Springer, 2017.

\bibitem{aajz} A. Alahmadi, H. Alsulami, S.K. Jain and E. Zelmanov, {\it Leavitt path algebras of finite Gelfand-Kirillov dimension}, J. Algebra Appl. $\mathbf{11}$ (6) (2012), 6pp.

\bibitem{aragoodearl} P. Ara, K. Goodearl, {\it Leavitt path algebras of separated graphs}, J. reine angew. Math. $\mathbf{669}$ (2012), 165--224.

\bibitem{ahls} P. Ara, R. Hazrat, H. Li, A. Sims, {\it Graded Steinberg algebras and their representations}, Algebra and Number theory, $\mathbf{12-1}$ (2018), 131--172.


\bibitem{pchr} G. Aranda Pino, J. Clark, A. an Huef,  I. Raeburn, {\it Kumjian--Pask
    algebras of higher-rank graphs}, Trans. Amer. Math. Soc. $\mathbf{365}$ (2013), no.
    7, 3613--3641.

\bibitem{carlsenruizsims} T.M. Carlsen, E. Ruiz, A. Sims, {\it Equivalence and stable isomorphism of groupoids, and diagonal-preserving stable isomorphisms of graph $C^*$-algebras and Leavitt path algebras}, Proc. Amer. Math. Soc. $\mathbf{145}$ (2017),1581--1592.


\bibitem{carlseneilers}  T.M. Carlsen, S. Eilers, E. Ortega, G. Restorff, {\it Flow equivalence and orbit equivalence for shifts of finite type and isomorphism of their groupoids,}  J. Math. Anal. Appl. $\mathbf{469}$ (2019), no. 2, 1088--1110. 














\bibitem{clarkhazratrigby}  L.O. Clark, R. Hazrat, S. Rigby, {\it Strongly graded groupoids and strongly graded Steinberg algebras}, arXiv:1711.04904

\bibitem{cep} L.O. Clark, R. Exel, E. Pardo, {\it A generalised uniqueness theorem and the graded ideal structure of Steinberg algebras}, Forum Mathematicum, $\mathbf{30}$ (2018), no. 3, 533--552.

\bibitem{cfst} L.O. Clark, C. Farthing, A. Sims, M. Tomforde, {\it A groupoid
    generalisation of Leavitt path algebras}, Semigroup Forum $\mathbf{89}$ (2014),
    501--517.

\bibitem{cm} L.O. Clark, C. Edie-Michell, {\it Uniqueness theorems for Steinberg
    algebras}, Algebr. Represent. Theory $\mathbf{18}$ (2015), no. 4, 907--916.



\bibitem{cs} L.O. Clark, A. Sims, {\it Equivalent groupoids have Morita equivalent
    Steinberg algebras}, J. Pure Appl. Algebra $\mathbf{219}$ (2015), 2062--2075.


\bibitem{cm1984} M. Cohen, S. Montgomery, {\it Group-graded rings, smash products, and
    group actions}, Trans. Amer. Math. Soc. $\mathbf{282}$ (1984), 237--258.



\bibitem{crainicmoerdijk}
M. Crainic, I. Moerdijk, {\it A homology theory for \'etale groupoids}, J. reine angew. Math. $\mathbf{521}$ (2000), 25--46. 










\bibitem{exel2008} R. Exel, {\it Inverse semigroups and combinatorial $C^*$-algebras}, Bull. Braz. Math. Soc. (N.S.) $\mathbf{39}$ (2) (2008) 191--313.



















\bibitem{gr} E.L. Green, {\it Graphs with relations, coverings and group-graded
    algebras}, Trans. Amer. Math. Soc. $\mathbf{279}$ (1983), no. 1, 297--310.

\bibitem{haefner} J. Haefner, {\it Graded equivalence theory with applications}, J. Algebra $\mathbf{172}$ (1995), 385--424.


\bibitem{roozbehhazrat2013} R. Hazrat, {\it The graded Grothendieck group and the classication of Leavitt path algebras}, Math. Annalen $\mathbf{355}$ (2013), 273--325.



\bibitem{grbook} R. Hazrat, Graded rings and graded Grothendieck groups,  London Math. Society Lecture Note Series $\mathbf{435}$, Cambridge University Press, 2016. 


\bibitem{hl} R. Hazrat, H. Li, {\it Graded Steinberg algebras and partial actions}, J. Pure App. Algebra (2018), $\mathbf{222}$ (2018) 3946--3967.




\bibitem{kp} A. Kumjian, D. Pask, {\it $C^{*}$-algebras of directed graphs and group actions}, Ergod. Th. Dynam. Systems $\mathbf{19}$ (1999), 1503--1519.
    
    
\bibitem{lindmarcus} D. Lind and B. Marcus. An introduction to symbolic dynamics and coding, Cambridge University
Press, 1995.


\bibitem{matui} H. Matui, {\it Homology and topological full groups of \'etale groupoids on totally disconnected spaces}, Proc. London Math. Soc. $\mathbf{104}$ (2012), 27--56.

\bibitem{matui15}  H. Matui, {\it Topological full groups of one-sided shifts of finite type}, J. Reine Angew. Math. $\mathbf{705}$ (2015) 35--84.

\bibitem{matui16} H. Matui, {\it \'Etale groupoids arising from products of shifts of finite type}, Advances in Mathematics $\mathbf{303}$ (2016) 502--548. 



\bibitem{nastasesu-vanoystaeyen} C. N\u{a}st\u{a}sescu, F. Van Oystaeyen, Methods of Graded Rings, Lecture Notes in Math, Vol 1836. Berlin: Springer-Verlag, 2004.








\bibitem{renault} J. Renault, A groupoid approach to $C^{*}$-algebras, Lecture Notes in Mathematics, 793. Springer, Berlin, 1980.



\bibitem{st} B. Steinberg, {\it A groupoid approach to discrete inverse semigroup
    algebras}, Adv. Math. $\mathbf{223}$ (2010), 689--727.









\bibitem{goehle} G. Goehle, {\it Groupoid crossed products, PhD thesis}, Dartmouth College, USA, 2009.

\bibitem{sw} A. Sims, D. P. Williams, {\it Renault'€™s equivalence theorem for reduced groupoid $C^*$-algebras}, J. Operator Theory $\mathbf{68}$ (2012), no. 1, 223--239.

\bibitem{we} S.B.G. Webster, {\it The path space of a directed graph}, Proc. Amer. Math.
    Soc. $\mathbf{142}$  (2014), 213--225.
    




\end{thebibliography}
\end{document}